\documentclass[onefignum,onetabnum]{siamart220329}

\usepackage{listings}
\usepackage{lipsum}
\usepackage{amsfonts}
\usepackage{graphicx}
\usepackage{epstopdf}
\usepackage{algorithmic}
\ifpdf
  \DeclareGraphicsExtensions{.eps,.pdf,.png,.jpg}
\else
  \DeclareGraphicsExtensions{.eps}
\fi

\usepackage[numbers,sort&compress]{natbib}

\newsiamremark{remark}{Remark}
\newsiamremark{hypothesis}{Hypothesis}
\crefname{hypothesis}{Hypothesis}{Hypotheses}
\newsiamthm{claim}{Claim}

\headers{High-Order-Integration on Regular Manifolds}{ G.~Zavalani, O.~Sander, and M.~Hecht }

\title{High-Order Integration on regular triangulated manifolds reaches Super-Algebraic Approximation Rates through Cubical Re-parameterizations
\thanks{Submitted to the editors DATE.
\funding{This work was partially funded by the Center of Advanced Systems Understanding (CASUS), financed by Germany's Federal Ministry of Education and Research (BMBF) and by the Saxon Ministry for Science, Culture and Tourism (SMWK) with tax funds on the basis of the budget approved by the Saxon State Parliament.}}}

\author{Gentian Zavalani \footnotemark[3]  \thanks{Center for Advanced Systems Understanding (CASUS), G\"{o}rlitz, Germany. \\
Email: g.zavalani@hzdr.de
}
\and 
Oliver Sander\thanks{Technische Universit\"{a}t Dresden,
Faculty of Mathematics, Dresden, Germany. \\ Email: oliver.sander@tu-dresden.de}
\and Michael Hecht\footnotemark[2] \thanks{ University of Wroc\l aw, Mathematical Institute,  Email: m.hecht@hzdr.de}}

\usepackage{amsopn}

\ifpdf
\hypersetup{
  pdftitle={Stable High-Order Approximation of Triangulated Manifolds, with Application to Surface Integration},
  pdfauthor={G. Zavalani, M. Hecht, and O. Sander}
}
\fi

\usepackage{amsfonts,amsmath,amssymb}
\usepackage{mathtools}
\usepackage{mathrsfs}
\usepackage{mathptmx}
\usepackage{booktabs}
\usepackage{protosem}
\usepackage{caption}
\usepackage{epigraph}
\usepackage{xcolor}
\usepackage{enumitem}
\usepackage{subcaption}
\usepackage{tikz}
\usepackage{soul}
\usepackage{float}
\usepackage[numbers,sort&compress]{natbib}
\definecolor{OurRed}{rgb}{0.64, 0.30, 0.30}

\usepackage{todonotes}

\usepackage{hyperref}

\DeclarePairedDelimiter{\norm}{\lVert}{\rVert}

\newcommand{\R}{\mathbb{R}}

\newcommand{\N}{\mathbb{N}}

\newcommand{\dist}{\mathrm{dist}}

\newcommand{\Cheb}{\mathrm{Cheb}}
\newcommand{\Oc}{\mathcal{O}}
\newcommand{\ee}{\varepsilon}
\newcommand{\li}{\left}
\newcommand{\re}{\right}

\newcommand{\vol}{\mathrm{vol}}

\DeclareMathOperator*{\argmin}{arg\,min}

\newcommand{\p}{\partial}

{\bf}{\it}

\begin{document}

\maketitle

\begin{abstract}
We present a novel methodology for deriving high-order volume elements (HOVE) designed for the integration of scalar functions over regular embedded manifolds.  For constructing  HOVE we introduce  \emph{square-squeezing}---a homeomorphic multilinear hypercube-simplex transformation---reparametrizing an initial flat triangulation of the manifold to a cubical mesh.
By employing square-squeezing, we approximate the integrand and the volume element for each hypercube domain of the reparameterized mesh through interpolation in Chebyshev–Lobatto grids. This strategy circumvents the Runge phenomenon, replacing the initial integral with a closed-form expression that can be precisely computed by high-order quadratures.

We prove novel bounds of the integration error in terms of the $r^\text{th}$-order total variation of the integrand and the surface parameterization, predicting high algebraic approximation rates that scale solely with the interpolation degree and not, as is common, with the average simplex size. For smooth integrals whose total variation is constantly bounded with increasing $r$, the estimates prove the integration error to decrease even exponentially,  
while mesh refinements are limited to achieve algebraic rates.
The resulting approximation power is demonstrated in several numerical experiments, particularly showcasing p-refinements to overcome the limitations of h-refinements for highly varying smooth integrals.
\end{abstract}


\section{Introduction}

Given a compact, orientable, $d$-dimensional $C^{r+1}$-manifold $S$, $r \geq 0$,
embedded into some $m$-dimensional Euclidean space $0 \leq d \leq m$, and an integrable function
$f : S \to \R$ with $f \in C^r$, this article proposes a novel surface integral algorithm approximating the integral
\begin{equation}\label{eq:task}
    \int_{S} f(\mathrm{x})\,dS\,.
\end{equation}
Such integrals appear in \emph{geometric processing}~\cite{Lachaud16}, \emph{surface--interface and colloidal sciences} \cite{zhou2005surface}, as well as optimization of production processes \cite{Gregg_1967,riviere2009}. Especially, they are central in 
many areas of applied numerical analysis, whereas
finite element (volume) methods \cite{dziuk2013finite,hansbo2020analysis, heine2004isoparametric} and spectral methods \cite{fortunato2022high, hale2016fast,townsend2016fast} exploit them to solve partial differential equations on curved surfaces \cite{grande2018analysis}.
While spectral methods are capable of realizing much higher order approximations than finite element methods \cite{Lloyd3, ruiz2018nonuniform,hale2014fast,fortunato2020fast,townsend2018fast,olver2020fast,olver2013fast}, current research aims to make them accessible for applications such as \emph{active morphogenesis} \cite{mietke2019self}, \emph{free-surface flows} \cite{nestler2018orientational}, or \emph{interfacial transport problems} \cite{marschall2017transport,jankuhn2018incompressible,YANG09}.


In contrast to integration tasks on flat domains, when integrating over an
embedded manifold, the additional challenge of approximating the embedding has to be addressed. To do so, we assume that the integrand $f$ is fully known in the sense
that we can evaluate it precisely and with reasonable computational cost
at any point $\mathrm{x} \in S$.
The construction of HOVE further assumes the existence of a triangulation
of $S$, i.e., a finite family $\{\varrho_i\}$ of differentiable maps of a reference
simplex $\triangle_d$ into $S$ such that the images partition $S$ up to a
zero set. With such a triangulation, the surface integral becomes a sum over the simplices
\begin{equation}\label{eq:P1}
\int_S f\,dS
=
\sum_{i=1}^K \int_{\triangle_d} f(\varrho_i(\mathrm{x}))
\sqrt{\det((D\varrho_i(\mathrm{x}))^T D\varrho_i(\mathrm{x}))}\,d\mathrm{x}.
\end{equation}

The integrand involves the Jacobians $D\varrho_i$
of the parametrization functions, as part of the
volume element. However, in many practical
applications, these Jacobians are not known with any
reasonable accuracy. One case where this happens
is when $S = l^{-1}(0)$ is the level-set of a $C^r$-function $l : \R^{d+1}\to \R$.
Then, the parametrizations $\varrho_i$ can be evaluated
by numerically looking for zeros of $l$. Typically, one starts with an
approximate triangulation by flat simplices
in the surrounding space and computes the values
$\varrho_i(\mathrm{x})$ by the closest-point projection
\cite{geometricpde, Demlow_Dziuk}.
While the implicit function theorem allows  to compute the derivatives $D\varrho_i(\mathrm{x})$,
this does not lead to high-order results;
see Remark~\ref{rem:affine_mesh_with_projections}
below.

We prove and numerically demonstrate that polynomial interpolation of the $\varrho_i$ yields a more powerful alternate solution.
When replacing both the integrand $f$ and the parametrizations $\varrho_i$ by polynomial
approximations $Qf$, $Q\rho_i$ (not necessarily of the same degrees),
the surface integral becomes
\begin{equation}
\label{eq:approximation_overview}
\int_S f\,dS
\approx
\sum_{i=1}^K \int_{\triangle_d} Qf(\varrho_i(\mathrm{x}))
\sqrt{\det((DQ\varrho_i(\mathrm{x}))^T DQ\varrho_i(\mathrm{x}))}\,d\mathrm{x}.
\end{equation}
The right-hand side is a closed form expression that, even though it includes the square-root function, can be precisely computed by 
standard simplex quadrature rules as long as the volume
element stays away from zero; see Corollary~\ref{cor:pull}.
The dominant part of the integration error is induced by the approximation error of the
interpolation operator $Q$.

Classic interpolation by piecewise polynomials using total $l_1$-degree polynomial spaces
on each simplex leads to approximation rates
that are only algebraic with increasing mesh size $h>0$,
i.e., the error behaves like $\Oc(h^k)$ for
some $k \ge 1$ \cite{Ciarlet02,shewchuk2002good}.
Such algorithms are particularly suitable for
integrands with limited regularity. We, however,
are primarily interested in the high-regularity case,
by which we mean the existence of an $r \gg 0$
such that $S \in C^{r+1}$ and $f \in C^r$
are of at most polynomially growing \emph{$r^{\text{th}}$ total variations} $V_{f,r}$,  $V_{\varrho_i,r}\in  o(r^k)$, for some $k \in \N$. Hereby,  we understand $V_{f,r}$ and $V_{\varrho_i,r}$
in the sense of Vitali and Hardy--Krause (see Definition~\ref{def:total_variation}).
In this setup, we obtain high algebraic up to  exponential approximation rates, with increasing interpolation degrees,  analogous rates for the polynomial derivatives, and consequently,  the integration task.

\subsection{Contribution}\label{sec:contr} Given a compact, orientable, $d$-dimensional $C^{r+1}$-manifold $S$, $r \geq 0$,
embedded into some $m$-dimensional Euclidean space $d \leq m$
and an integrand $f : S \to \R$ in~$C^r$.

\begin{enumerate}[left=0pt,label=\textbf{C\arabic*)}]
\item \label{novelty:approximation}
We provide a novel method for approximating $S$
by a piecewise polynomial manifold.
Given a flat triangulation $T$ of $S$, on each simplex
we reparametrize by a particular hypercube--simplex
transformation $\sigma_* : \square_d \to \triangle_d$,
we term \emph{square-squeezing}. We then interpolate the $\varrho_i$ for each hypercube in
$k^\text{th}$-order tensorial Chebyshev--Lobatto nodes.
As well known, this avoids
Runge's phenomenon for regular interpolation tasks and have the advantage that the FFT is available for an $\mathcal{O}(N \log N)$ implementation of the differentiation process, and they also have slight advantages connected to their ability to approximate functions.

\item \label{novelty:quadrature}
Given \ref{novelty:approximation}, arbitrarily high-order volume elements (HOVE) can be constructed for each cube. When integrating  scalar functions
$f : S \to \R$, this results in numerical errors, rapidly decreasing with the order of the applied quadrature rule.
Possible options are tensorial Gauss--Legendre rules
or pull-backs of the symmetric
Gauss simplex rules~\cite{dunavant1985high}, whereas the latter
are more efficient, see Corollary~\ref{cor:pull}.

\item \label{novelty:error_bound}
In Theorem~\ref{Main.Thm} we prove a novel estimate for the
error $E(f,S)$ of HOVE.
Specifically, we show that
\begin{equation}\label{eq:est}
E(f,S)  \leq C\big(n^{-(r-d+1)}  + k^{-(r-d-1)}\big)\,, \quad  C=C(V_{f,r}, V_{\varrho_i,r}, S)>0\,,
\end{equation}
where $k$ is the polynomial degree used for approximating the geometry, while $n$ denotes the polynomial degree employed for interpolating the integrand $f$.
The constant $C$ depends
on the surface $S$ and on the $r^{\text{th}}$
total variations $V_{f,r}$, $V_{\varrho_i,r}$
of  the integrand $f$ and the parameterizations $\varrho_i$, respectively.
To the best of our knowledge, this estimate is the first one guaranteeing
convergence to the correct integral when increasing
the polynomial order alone.
If, in addition, the  $r^{\text{th}}$ total variations are uniformly bounded
\begin{equation}\label{eq:VB}
    \limsup_{r \to \infty} V_{f,r} < \infty \,,
    \qquad
    \limsup_{r \to \infty} V_{\varrho,r} < \infty \,,
\end{equation}
Equation~\eqref{eq:est} even implies that the integration error decreases exponentially
\begin{equation}
E(f,S)  \leq CR^{-\min\{n,k\}}\,, \quad\quad \text{for some} \quad  R>1 \,,  \quad C=C(f,S)>0\,.
\end{equation}
\end{enumerate}

We want to stress that prior estimates
for alternative surface quadrature methods  \cite{Demlow09, dziuk2013finite,Zavalani23} show only
\begin{equation}\label{eq:est2}
E(f,S)  \leq C(h^{n+1} + h^{k+1})\,, \quad  C=C(n,k,f,S)>0\,,
\end{equation}
where $h>0$ is the mesh size.
Here, the constant $C$ explicitly depends on the degrees $n$ and $k$.
Since
potentially $C(n,k,f,S) \to \infty$ with $n,k \to \infty$, in contrast to \eqref{eq:est}, no guarantees of higher accuracy or even convergence is given for $p$-refinements, increasing $n$, $k$.
Moreover, the approximation rate is only algebraic in the mesh size $h$.

Experiments in Section~\ref{sec:NUM} show the super-algebraic or even exponential approximation rates predicted by \eqref{eq:est}, suggesting HOVE to be the superior choice for regular integration tasks. In particular, HOVE resolves integration tasks of high variance, Section~\ref{sec:BC}, that are non-reachable by low--order methods, even when exploiting super-resolution meshes, potentially generated by $h$-refinements.

At this moment, HOVE is limited to scenarios where the integrand $f: S\to \R$ can be evaluated at any point $x \in S$, and the manifold $S$ is (implicitly) parameterized.
In our concluding thoughts, Section~\ref{sec:CON}, we sketch how recent results \cite{GPLS,REG_arxiv,drake2022implicit} allow to
overcome this limitation, making HOVE applicable for non-parametrized surfaces
and functions given only in specific sample points.

\subsection{Related work}\label{sec:relW}
The importance of computing integrals on manifolds
is reflected  in the large number of articles addressing this subject. Approaches might be divided into mesh-free  methods, requiring a partition of unity, and mesh-based methods. A comprehensive review of the entire literature is beyond the scope of this article. The following list highlights specific contributions that may directly relate to or complement our work.

\begin{enumerate}[left=0pt,label=\textbf{R\arabic*)}]
\item The strength of mesh-free approaches, such as \emph{moving least squares}, comes from their ability to approximate integrals with discontinuities for arbitrary  function data. However, limitations are the stability of the involved regression methods~\cite{platte:2011} and the computational cost for computing a proper partition of unity. We recommend Belytschko et al.  \cite{belytschko1996} for an excellent survey on the subject.

\item\label{Ray} Ray et al.~\cite{Ray2012} realise
\emph{High-Order Integration over Discrete (Triangulated) Surfaces} (IDS)
based on stabilized least squares, deriving $k^{\text{th}}$-order surface approximations. \linebreak
While the stabilized least-square regression avoids Runge's phenomenon the computational costs rapidly increase with the order of the approximation.
Recent extensions \cite{li2021compact}
address the task of computing integrals over non-parametrized surfaces.

\item Piecewise polynomial approximations of regular hyper-surfaces $S = l^{-1}(0)$
in $\R^3$ are studied by Dziuk and Elliot~\cite{dziuk2013finite}. Realizations are given by Demlow \cite{Demlow09}, Chien and Atkinson \cite{Chien1995, atkinson1995piecewise}, and Praetorius and Stenger  \cite{CurvedGrid}.
However, all approaches rest on interpolation in equidistant nodes
on simplices.
Consequently, they are sensitive to Runge's phenomenon and become unstable for high orders. An extended investigation of the error analysis provided by \cite{Demlow09, dziuk2013finite, CurvedGrid} is given in \cite{Zavalani23}.

\item \label{reeger} Reeger et al.~\cite{reeger2016numerical} propose to use local radial basis function-generated finite differences (RBD-FD) for efficiently generating quadrature weights for arbitrary node sets. This enables to approximate surface integrals for any given function data.
\end{enumerate}

While \ref{Ray} and \ref{reeger} address the harder problem of integrating $f$ based on samples of $f$
given at particular point sets,
even in the case of regular surface integrals all approaches are limited to achieve prior specified algebraic approximation rates.
In contrast, we prove HOVE's integration rates to be of high algebraic order, specifically depending on the instance's total variation, resulting in  super-algebraic up to exponential convergence for variationally bounded integration tasks.

\subsection{Notation} Throughout the article,
we denote with $\square_d=[-1,1]^d$ the closed $d$-dimensional standard hypercube, 
and with $\triangle_d = \{\mathrm{x} \in \R^d: x_1,\dots,x_d \geq 0 \,,\sum_{i=1}^d|x_i|\leq 1\}$ the standard $d$-simplex in $\R^d$.
For a set $U \subseteq \R^d$ we denote with $\mathring U$ its interior, with $\overline{U}$ its closure, and with $\p U = \overline{U}\setminus U$ its boundary.
The canonical basis of $\R^d$ is called $\{e_i\}_{i=1,\dots,d}$.
For vectors $\mathrm{x}, \mathrm{y} \in \R^d$ we denote by $\langle\mathrm{x},\mathrm{y}\rangle$ the standard Euclidean inner product and by $\|\mathrm{x}\|$
the corresponding norm.
Furthermore, we set
$\|\mathrm{x}\|_p = (\sum_{i=1}^d |x_i|^p)^{1/p}$ ,$1 \leq p < \infty$, the $l_p$-norm,
and $\|\mathrm{x}\|_\infty = \max_{i=1,\dots,d}|x_i|$.

Lets denote with $\N_{0}^d:=(\N\cup\{0\})^d$, we define monomials as $\mathrm{x}^{\boldsymbol{\alpha}} = \prod_{i=1}^d x_i^{\alpha_i}$, $\mathrm{x} \in \R^d$, $\boldsymbol{\alpha} \in \N_0^d$ and 
consider multi-index sets $A_{d,n,p}=\{\boldsymbol{\alpha} \in \N_{0}^d : \|\boldsymbol{\alpha}\|_p \leq n\}$, $1 \leq p \leq \infty$, inducing the real 
polynomial vector spaces $\Pi_{d,n,p} = \mathrm{span}\{\mathrm{x}^{\boldsymbol{\alpha}}\}_{\boldsymbol{\alpha} \in A_{d,n,p}}$
of $l_p$-degree $n$. 
In contrast to the common \emph{total $l_1$-degree}
polynomial space (also known as the full polynomial space),
the vector space of all real polynomials of \emph{maximum $l_\infty$-degree} $n$ in $d$ variables
will be central.  We will denote this space by
$\Pi_{d,n} = \Pi_{d,n,\infty} \mathrm{span}\{\mathrm{x}^{\boldsymbol{\alpha}}\}_{\boldsymbol{\alpha}\in A_{d,n}}$, $A_{d,n} =A_{d,n,\infty}$.

By $L^2(\square_d)=\{f : \square_d \to \R : \int_{\square_d}|f(\mathrm{x})|^2d\mathrm{x} < \infty\}$
we denote the Hilbert space of
square-Lebesgue-integrable functions.
The Banach space
of $r$-times continuously differentiable functions on $\mathring \square_d$
will be called $C^r(\square_d)$, $r\in\N$,
with norm
\begin{equation}
    \|f\|_{C^r(\square_d)}
    = \sum_{\substack{\boldsymbol{\alpha} \in \N^d \\ \|\boldsymbol{\alpha}\|_1\leq k}} \; \sup_{\mathrm{x}\in \mathring \square_d}|\p^{\boldsymbol{\alpha}} f(\mathrm{x})|\,, \quad \p^{\boldsymbol{\alpha}} f(\mathrm{x}) = \p_{x_1}^{\alpha_1}\cdots\p_{x_d}^{\alpha_d} f(\mathrm{x})\,.
\end{equation}

Finally, we introduce the main vehicle to quantify
regularity of the integrands used in this work.
\begin{definition}[$r^\text{th}$-order total variation]
\label{def:total_variation}
Let $r \geq 0$, $f: \square_d \to \R$ and its derivatives through $\p^{\boldsymbol{\beta}} f$, $\boldsymbol{\beta} = (r+1,\dots,r+1)$ be absolutely continuous (differentiable almost everywhere). We define the $r^{\text{th}}$ total variation $V_{f,r}$ as
\begin{equation}\label{eq:var}
   V_{f,r}=\max_{\substack{\boldsymbol{\beta} \in \N^d \\ \|\boldsymbol{\beta}\|_\infty \leq r+1}}\int_{\square_d} |\p^{\boldsymbol{\beta}}f(\mathrm{x})| d\mathrm{x} \,, 
\end{equation}
and refer $f$ as having bounded $r^{\text{th}}$ total variation, whenever $V_{f,r} < \infty$ exists.
\end{definition}
This definition recaptures the notion of Vitali and Hardy--Krause \cite{Clarkson1933OnDR,Aistleitner14,Owen2004}.

\section{Integrals based on triangulations}
Simplex meshes are 
typically much easier to obtain in practice~\cite{Persson} than cube meshes, 
consequently serving as our starting point here.

\subsection{Nonconforming simplex meshes}

Integrals over a manifold $S$ can be rewritten as integrals
over simplices if the manifold is equipped with a triangulation.

\begin{definition}[Nonconforming triangulation]
\label{def:triangulation}
We call nonconforming triangulation of $S$ a finite family
of maps $\varrho_i$ and corresponding sets $V_i \subset S$, $ i=1,\dots,K$
such that
\begin{equation*}
    \varrho_i : \triangle_d \to V_i \subseteq S\,, \quad \bigcup_{i=1}^K
    \overline{V_i} = S\,,
    \quad \bigcap_{i\neq j} V_i \cap \ V_j = \emptyset\,,
\end{equation*}
and the restrictions of the $\varrho_i$ to the interior
$\mathring \triangle_d$ are diffeomorphisms.
\end{definition}

\begin{remark}
Note that we do not require compatibility conditions between
adjacent simplices, which makes this notion of triangulation more general than
the common one from~\cite{thurston2014three}.
\end{remark}

For immersed manifolds $S \subset \R^m$ we will write
$D\varrho_i(\mathrm{x}) : \R^d \to \R^m$ for the
Jacobian of the parametrization $\varrho_i$ at $\mathrm{x}$, enabling to   compute integrals
simplex by simplex.

\begin{lemma}\label{prob:1}
Given a nonconforming triangulation of $S$, the integral
of an integrable function $f:S \to \R$ is
\begin{equation}\label{eq:SI}
    \int_S f dS = \sum_{i=1}^K\int_{\triangle_d} f(\varrho_i(\mathrm{x})) g_i(\mathrm{x})\,d\mathrm{x}\,,
\end{equation}
where $g_i(\mathrm{x})=\sqrt{\det((D\varrho_i(\mathrm{x}))^T D\varrho_i(\mathrm{x}))}$
is the volume element.
\end{lemma}

\begin{remark}[Closest-point projections]
\label{rem:affine_mesh_with_projections}
In practice, triangulations of an embedded
manifold $S$ are frequently given as a set of flat simplices
in the embedding space $\R^m$, together with local
projections from these simplices onto $S$ (Fig.~\ref{fig:app_frame}).  More formally, let
\begin{equation}\label{eq:triang}
 T_i \subseteq \R^m\,,
 \qquad
 i = 1,\dots,K
\end{equation}
be a set of $d$-simplices.
For each simplex $T_i$ we assume that there is a well-defined $C^{r+1}$-embedding $\pi_i : T_i \to S$ and an invertible affine transformation
$\tau_i : \triangle_d \to T_i$, such that the maps $\varrho_i=\pi_i \circ \tau_i : \triangle_d \to S$ form a triangulation in the sense of Definition~\ref{def:triangulation}. Commonly, \emph{the closest-point
 projection}
\begin{equation*}
    \pi^* : \mathcal{N}_{\delta}(S) \to S\,, \quad \pi^*(\mathrm{x}) = \argmin_{\mathrm{y} \in S}\dist(\mathrm{x},\mathrm{y})
\end{equation*}
serves as a realisation of the $\pi_i$.
Recall from \cite{geometricpde, Demlow_Dziuk} that
given an open neighborhood $\mathcal{N}_{\delta}(S)=\{\mathbf{x}\in \mathbb R^{m} :  \mathrm{dist}(\mathbf{x},S)<\delta\}$ of a $C^{r+1}$-surface $S$, $r \geq 2$ with $\delta$ bounded by the reciprocal of the maximum of all principal curvatures on $S$, the closest-point projection
is well-defined on $\mathcal{N}_{\delta}(S)$ and of regularity $\pi^*\in C^{r-1}(T,S)$. 

In practice, $\pi^*$ is usually approximated by $\pi^*(\mathrm{x})\approx\mathrm{x}-\mathrm{sd}(\mathrm{x})\eta(\mathrm{x})$ with $\mathrm{sd}$ being the \emph{signed} distance function to $S$ and $\eta(\mathrm{x})$ the normal field, extended to $\mathcal{N}_{\delta}(S)$.
In such cases, the Jacobian $D\pi^*$ is highly sensitive to the 
approximation quality of the normal field $\eta$. Apart from standard  
cases (e.g., spheres and tori), where $\eta$ is known analytically, high-order approximates of  $D\pi^*$ cannot be derived by this approach.
\end{remark}

\subsection{Re-parametrization over cubes}

The main difficulty in providing a numerical approximation of
\eqref{eq:SI} is obtaining the unknown derivatives $D\varrho_i$ that appear in the volume element.
One classic approach, followed also by \cite{dziuk2013finite}, is to replace the Jacobians
$D\varrho_i$ by the Jacobians of a polynomial
approximation, typically obtained by interpolation
on a set of interpolation nodes in $\triangle_d$.
However, the question of how to distribute nodes in simplices in order to enable stable high-order polynomial interpolation is still not fully answered \cite{CHEN1995405,taylor2000generalized}.

To circumvent these limitations, we instead propose to
re-parametrize the curved simplices $S_i$
over the $d$-dimensional hypercube $\square_d$.

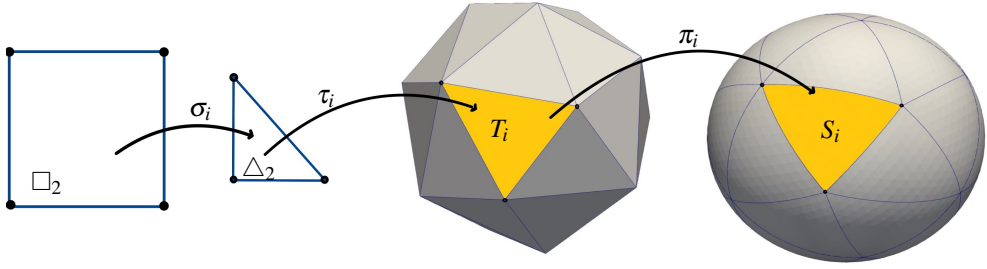
\begin{figure}[t!]
    \centering
    \begin{tikzpicture}
        \node[inner sep=0pt] at (0,0) {\includegraphics[clip,width=1.0\columnwidth]{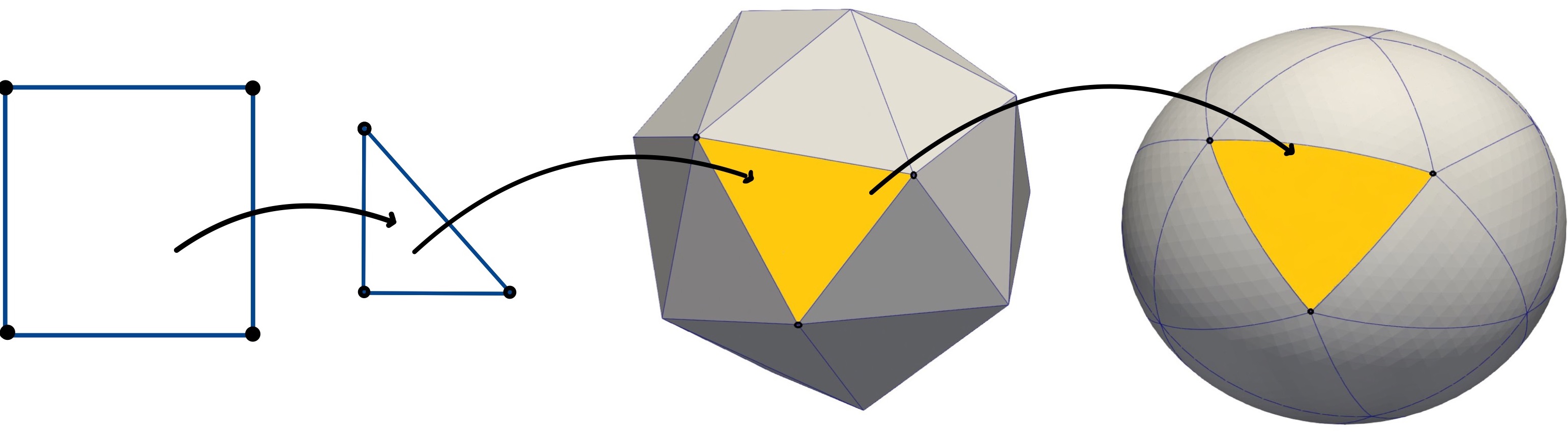}};
        
    
\node[anchor=north west] at (-0.2,0.3) {
$T_i\,$};
        \node[anchor=north west] at (2.3,1.5) {$\pi_i$};

        \node[anchor=north west] at (4.2,0.3) {$S_i\,$};

        \node[anchor=north west] at (-2.5,0.7) {$\tau_i$};
        \node[anchor=north west] at (-3.5,-0.20) {$\triangle_2$};
        \node[anchor=north west] at (-4.2,0.52) {$\sigma_i$};
        \node[anchor=north west] at (-6.3,-0.4) {$\square_2$};
    \end{tikzpicture}
    \caption{Construction of a surface parametrization over $\triangle_2$ by
closest-point projection from a piecewise affine approximate
mesh, and re-parametrization over the square $\Box_2$.}
    \label{fig:app_frame}
\end{figure}

\begin{definition}[Re-parametrization over cubes]
\label{def:cubical_reparametrization}
Let $\sigma : \square_d \to \triangle_d$ be a 
homeomorphism whose restriction
$\sigma_{| \mathring \square_d} : \mathring \square_d \to \mathring \triangle_2$ to the interior is a $C^r$-diffeomorphism, $r\geq 0$.  We call
\begin{equation}\label{eq:mesh}
\varphi_i : \square_d \to S\,,
\quad
\varphi_i = \varrho_i \circ \sigma = \pi_i \circ \tau_i\circ \sigma\,,
\quad
i=1,\dots,K
\,,
\end{equation}
a $r$-regular cubical re-parametrization whenver the coordinate functions of $\varphi_i$
are of bounded $r^{\text{th}}$ total variation, Definition~\ref{def:total_variation}, for all $i=1,\dots,K$.
\end{definition}

With such a re-parametrization, we effectively have
a hypercube mesh along with our simplex one, enabling us to construct
geometry approximations as described below.

\subsection{The square-squeezing re-parametrization map}

For the hypercube--simplex re-parametrization,
we propose to use the following multilinear map.

\begin{figure*}[!t]
  \begin{subfigure}{0.33\textwidth}
  \centering
    \includegraphics[width=1.0\linewidth]{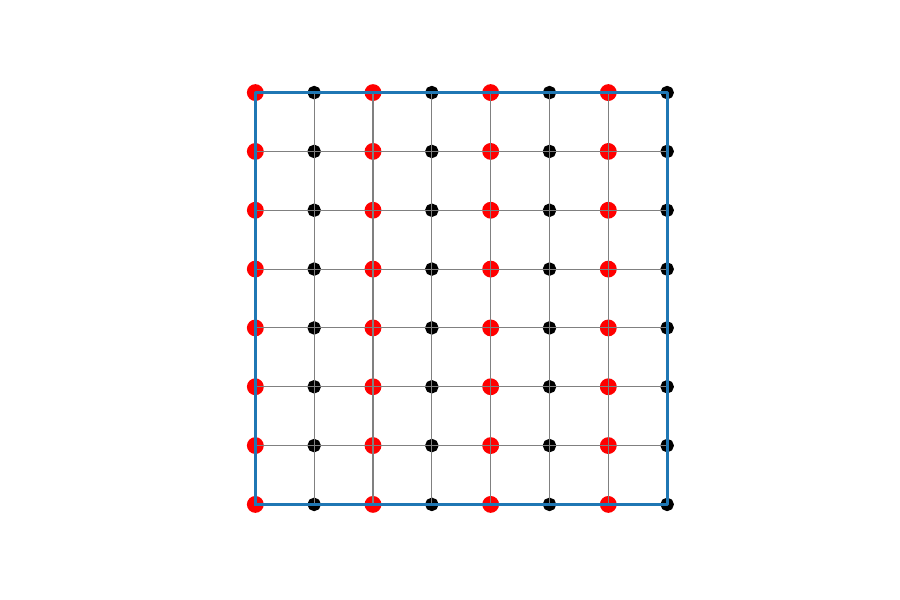}
  \end{subfigure}%
  \hfill
  \begin{subfigure}{0.33\textwidth}
  \centering
    \includegraphics[width=1.0\linewidth]{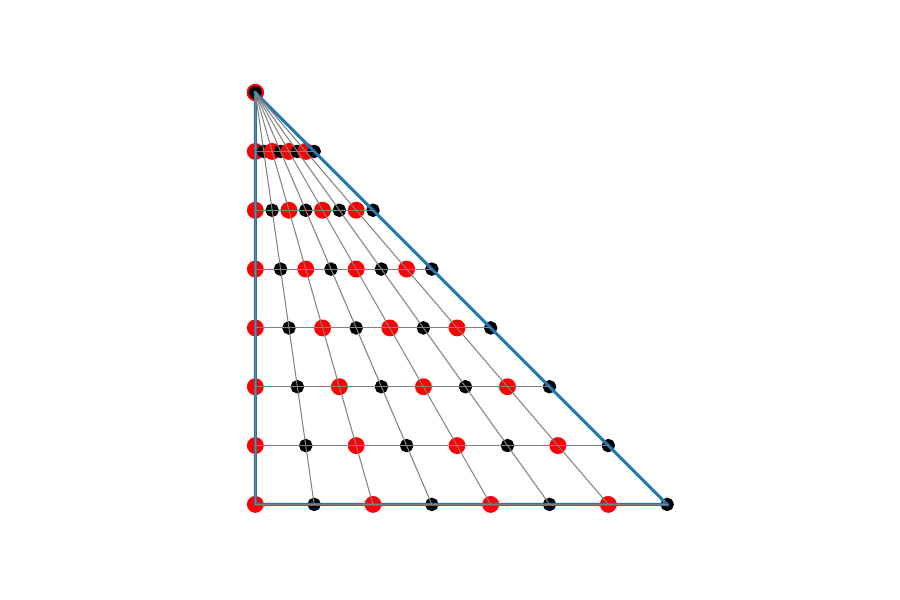}
  \end{subfigure}%
  \hfill
  \begin{subfigure}{0.33\textwidth}
  \centering
    \includegraphics[width=1.0\linewidth]{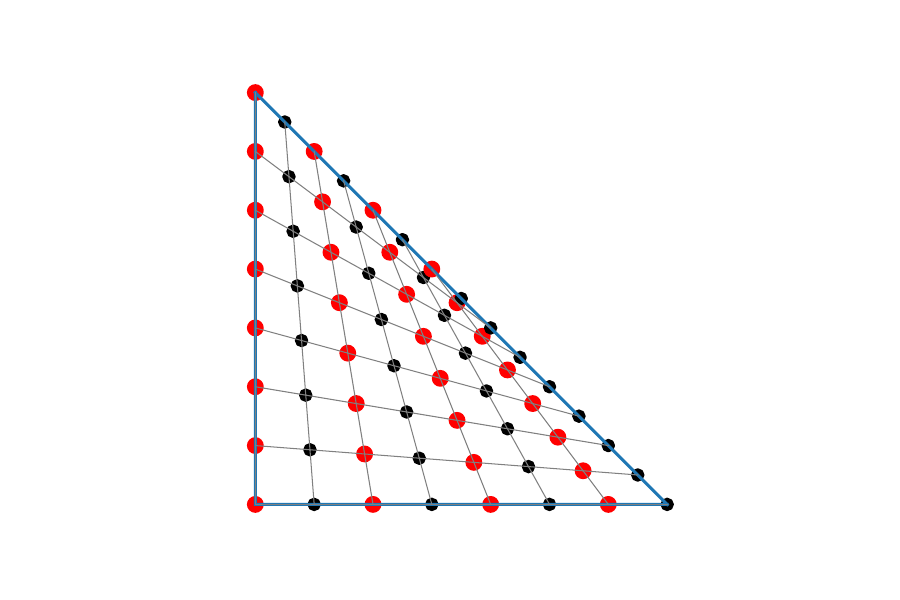}
  \end{subfigure}

  \begin{subfigure}{0.33\textwidth}
  \centering
    \includegraphics[width=1.0\linewidth]{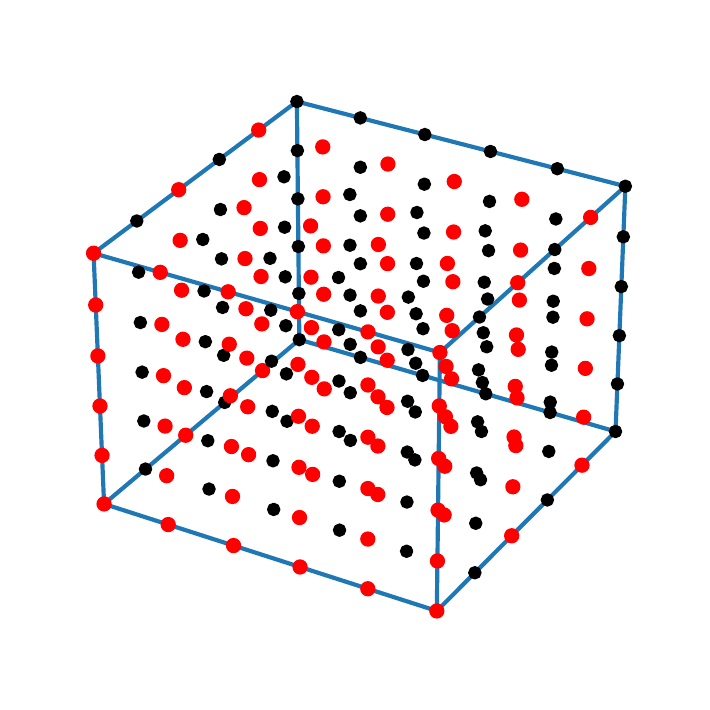}
    \caption{Standard hypercube}
    \label{fig:hypercube_grid}
  \end{subfigure}%
  \hfill
  \begin{subfigure}{0.33\textwidth}
  \centering
    \includegraphics[width=1.0\linewidth]{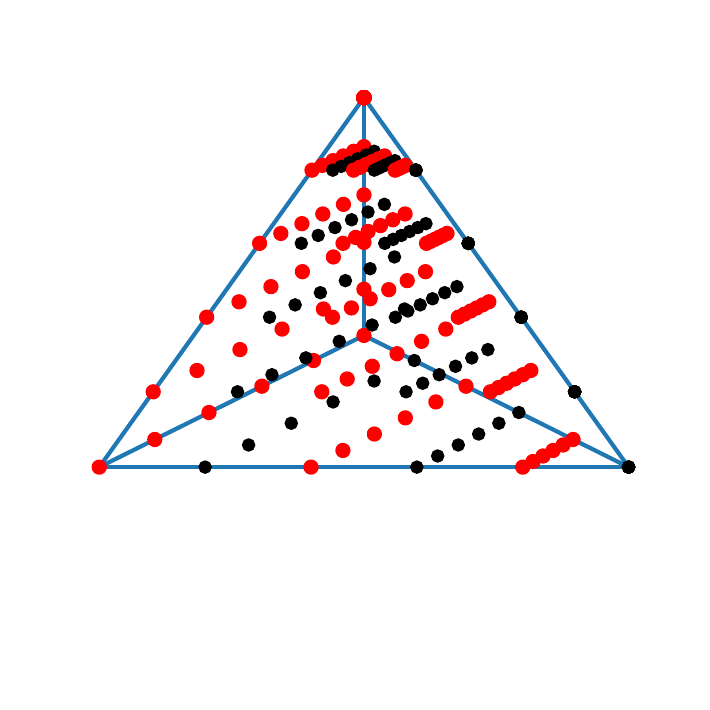}
    \caption{Duffy's transformation}
    \label{fig:duffy_grid}
  \end{subfigure}%
  \hfill
  \begin{subfigure}{0.33\textwidth}
  \centering
    \includegraphics[width=1.0\linewidth]{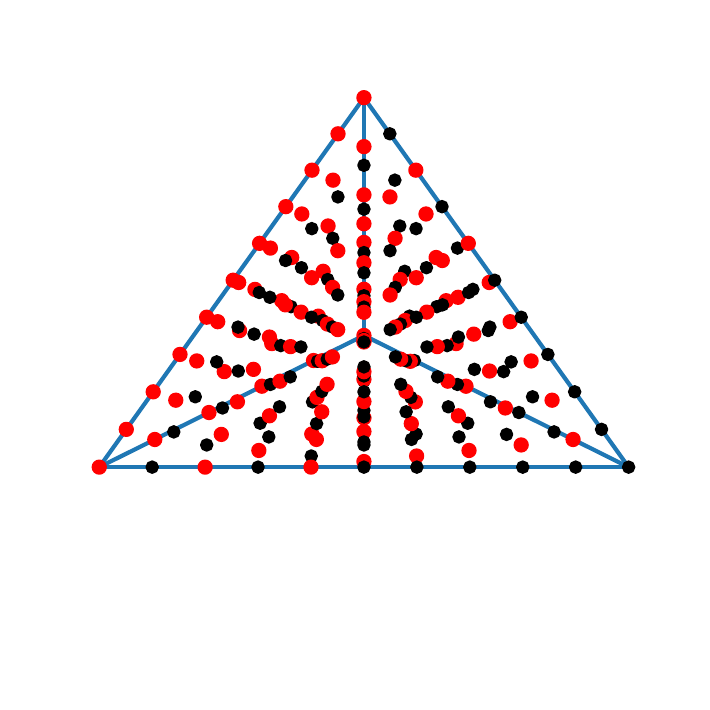}
    \caption{Square-squeezing}
    \label{fig:square_squeezing_grid}
  \end{subfigure}
  \caption{Multi-linear cube--simplex transformations for $d=2$ and $d=3$:
  Deformations of equidistant grids(\ref{fig:hypercube_grid}),
  under Duffy's transformation (\ref{fig:duffy_grid}),
  and square-squeezing (\ref{fig:square_squeezing_grid})}
  \label{fig:equi_tetra-cube3}
\end{figure*}

\begin{definition}[Square-squeezing]\label{def:ss}
Let $[0,1]^d$ be the $d$-dimensional unit cube,
with vertex set $A_{d,2}$. We call \emph{square-squeezing}
the map $\sigma_* : [0,1]^d \to \triangle_d$ that maps
the corners $\gamma \in A_{d,2} = \{ 0,1\}^d$ to
\begin{equation*}
\sigma_*(\gamma)
=
\begin{cases}
(0,\dots,0) & \text{if $\gamma = (0,\dots, 0)$} \\
\frac{\gamma}{\norm{\gamma}_1} & \text{otherwise},
\end{cases}
\end{equation*}
and uses multilinear interpolation for the rest of the domain.
\end{definition}

Note that all vertices of the simplex are mapped to themselves.
In other words:
$\sigma_* :   [0,1]^d \to \triangle_d$, $d \in \N$ is given by
\begin{equation}\label{eq:MLS}
    \sigma_*(\mathrm{x}) 
    =
    \sum_{\gamma \in A_{d,2}} \frac{\gamma}{|\gamma|} \Phi_\gamma\,, \quad  \Phi_\gamma = \prod_{i=1}^d x_i^{\gamma_i}(1-x_i)^{1-\gamma_i}\,.
\end{equation}
Since, this article operates on the standard cube $\square_d =[-1,1]^d\neq [0,1]^d$, we re-scale $\mathrm{x} \mapsto \tilde{\mathrm{x} }= (x_1 +1, \dots, x_m +1)/2$ for defining $\sigma_*(\mathrm{x})= \sigma_*(\tilde{\mathrm{x}})$ on $\square_d$. 
For illustration, we consider the important two-dimensional case in more detail:

\begin{remark}[Square-squeezing in two dimensions]\label{rem:ss}
We re-scale $\square_2$ to $[0,1]^2$ by setting $\tilde x_1 = (x_1+1)/2$, $\tilde x_2 = (x_2+1)/2$.
The square-squeezing transformation on $[0,1]^2$ becomes
\begin{equation}\label{eq.rec-tri}
   \sigma_* : [0,1]^2 \to \triangle_2\,,
    \quad
    \sigma_*(\tilde x_1, \tilde x_2)
    =
    \Big(\tilde x_1-\frac{\tilde x_1 \tilde x_2}{2},\; \tilde x_2-\frac{ \tilde x_1 \tilde x_2}{2}\Big)^T\,.
\end{equation}
The inverse map $\sigma^{-1}_* : \triangle_2 \to \square_2$ is given by
\begin{equation}\label{eq:inv_ss}
\sigma_*^{-1}(u,v)
=
\begin{pmatrix}
1+\left(u-v\right)-\sqrt{\left(u-v\right)^2+4\left(1-u-v\right)} \\
1-\left(u-v\right)-\sqrt{\left(u-v\right)^2+4\left(1-u-v\right)}
\end{pmatrix}.
\end{equation}
Both $\sigma_*$ and $\sigma_*^{-1}$ are continuous 
on $\square_d,\triangle_d$, respectively, showing  
square-squeezing to be a homeomorphism. The square-root term in \eqref{eq:inv_ss} is
smooth for all $(u,v)$ in $\triangle_2\setminus \{(\frac 12, \frac 12)\}$.
Hence, the restriction to the interior $\sigma_{* \, | \mathring \square_2} : \mathring \square_2 \to \mathring \triangle_2$
is a diffeomorphism.
Further, it is easy to show that $\|D\sigma_*\|_{C^0(\square_2)} \leq 1$. 
\end{remark}

\begin{remark}[Square-squeezing in three dimensions]
In dimension $d=3$, we term  $\sigma_* : \square_3 \to \triangle_3$, $(x,y,z)\mapsto (u,v,w)$ cube-squeezing, visualized in  Fig.~\ref{fig:equi_tetra-cube3}, and  explicitly given when re-scaling to $[0,1]^3$ by
\begin{equation*}
\sigma_* : [0,1]^3 \to \triangle_3,
\qquad
\sigma_*(\tilde x_1,\tilde x_2,\tilde x_3) = 
\begin{pmatrix}
\displaystyle (\tilde x_1-\frac{\tilde x_1 \tilde x_2}{2})(1-\frac{\tilde x_3}{2} +\frac{\tilde x_2 \tilde x_3}{6})
\vspace{8pt} 
\\
\vspace{8pt} 
\displaystyle (\tilde x_2-\frac{\tilde x_1 \tilde x_2}{2})(1-\frac{\tilde x_3}{2} +\frac{\tilde x_1 \tilde x_3}{6})\\
\vspace{8pt} 
\displaystyle (\tilde x_3-\frac{\tilde x_3 \tilde x_1}{2})(1-\frac{\tilde x_2}{2} + \frac{\tilde x_1 \tilde x_2}{6})
\end{pmatrix}.
\end{equation*}
\end{remark}

\begin{remark}\label{rem:duff} Note that
the commonly used \emph{Duffy transformation} \citep{Duffy82}
\begin{equation}\label{eq.duffy}
    \sigma_{\text{Duffy}} : \square_2 \to \triangle_2\,,
    \quad
    \sigma_{\text{Duffy}}(x,y) = \Big(\frac{1}{4}\left(1+x\right)\left(1-y\right),\frac{1+y}{2}\Big)\,,
\end{equation}
collapses one entire edge
of the square to the single vertex $(0,1)$. Thus, $\sigma_{\text{Duffy}}$
is a homeomorphism between $\mathring \square_2$ and
$\mathring \triangle_2$, but not between $\square_2$
and $\triangle_2$. Consequently, $\sigma_{\text{Duffy}}$ can only transform interpolation
or quadrature nodes from $\square_2$ to 
$\triangle_2$ and back if none of these nodes
is on the collapsed edge of $\square_2$
or the point $(0,1)$ of $\triangle_2$, excluding the case of 
\emph{Chebyshev-Lobatto nodes} \eqref{eq:CHEB} that are commonly considered as the optimal choice for interpolation tasks on hypercubes.

\end{remark}

\section{Approximation theory on hypercubes}

We now construct stable polynomial approximations of the
geometry functions $\varrho_i : \triangle_d \to \R^m$.
For this, we re-parametrize them to functions
on the cube $\varphi_i = \varrho_i \circ \sigma : \square_d \to \R^m$,
and approximate those using interpolation with
tensor-product polynomials. The resulting approximation
can be pulled back to the triangle domain
via $\sigma^{-1}: \triangle_d \to \square_d$.

\subsection{Interpolation in the hypercube}\label{sec:IP}

Throughout this section, $f$ is a generic
function on the standard square $\square_d$.
Afterwards, $f$ may play the role of  the coordinate functions of 
the geometry parametrizations $\varphi_i$
or pull-backs $f \circ \varphi_i$ of the integrand $f : S \to \R$.


We consider interpolation in tensorial grids.
\begin{definition}[Interpolation grid]
\label{def:EA}
For numbers $d, n \in \N$ let $P_1,\dots,P_d \subseteq [-1,1]$
be sets of size  $|P_i|=n+1$ each.
Then we call $G_{d,n}= \bigoplus_{i=1}^d P_i$ an interpolation grid.
For any multi-index ${\boldsymbol{\alpha}} \in A_{d,n}$
we denote with $p_{\boldsymbol{\alpha}}= (p_{\alpha_1,1},\dots,p_{\alpha_d,d}) \in G_{d,n}$, $p_{\alpha_i,i} \in P_i$, 
the corresponding grid node of $G_{d,n}$.
\end{definition}

We use such a grid to define the corresponding
interpolation operator $Q_{G_{d,n}}  : C^0(\square_d) \to \Pi_{d,n}$, $f \mapsto Q_{G_{d,n}} f$, uniquely determined by $Q_{G_{d,n}} f(p_{\boldsymbol{\alpha}}) = f(p_{\boldsymbol{\alpha}})$ 
for all $p_{\boldsymbol{\alpha}} \in G_{d,n}$. For an explicit representation, we generalize
one-dimensional Newton and Lagrange interpolation
to multivariate interpolation
on the grids $G_{d,n}$ \cite{cohen2,cohen3,PIP1,PIP2,MIP,IEEE}.

\begin{definition}[Lagrange and Newton polynomials]
\label{def:LN}
Let $G_{d,n} = \bigoplus_{i=1}^d P_i$ be an interpolation grid
indexed by a multi-index set $A_{d,n}$.
For each $\boldsymbol{\alpha} \in A_{d,n}$ the tensorial multivariate Lagrange polynomial is
\begin{equation}\label{eq:L}
  L_{\boldsymbol{\alpha}}(x)= \prod_{i=1}^d l_{\alpha_i,i} (x) \,, \quad l_{j,i} (x) = \prod_{k=0, k \not = j}^n \frac{x_i-p_{k,i}}{p_{j,i} - p_{k,i}}  \,.
\end{equation}
The $\boldsymbol{\alpha}$-th tensorial multivariate Newton polynomial is
\begin{equation}\label{eq:N}
 N_{\boldsymbol{\alpha}}(x) = \prod_{i=1}^d \prod_{j=0}^{\alpha_i}(x_i-p_{j,i})\,,\quad p_{j,i} \in P_i\,.
\end{equation}
\end{definition}
Both the Lagrange and Newton polynomials form bases of the polynomial space $\Pi_{d,n}$ induced by $A_{d,n}$.
As the $L_{\boldsymbol{\alpha}}$ satisfy
$ L_{\boldsymbol{\alpha}} (p_{\boldsymbol{\beta}}) = \delta_{\boldsymbol{\alpha},\boldsymbol{\beta}}$
for all $\boldsymbol{\alpha} \in A_{d,n}$, $p_{\boldsymbol{\beta}} \in G_{d,n}$
we deduce that
given a function $f :\square_
d\to \R$, the interpolant $Q_{G_{d,n}}f \in \Pi_{d,n}$
can be computed as
\begin{equation}\label{eq:NEWT}
    Q_{G_{d,n}}f = \sum_{\boldsymbol{\alpha} \in A_{d,n}} f(p_{\boldsymbol{\alpha}})L_{\boldsymbol{\alpha}} = \sum_{\boldsymbol{\alpha} \in A_{d,n}}b_{\boldsymbol{\alpha}} N_{\boldsymbol{\alpha}}\,,
\end{equation}
where the coefficients $b_{\boldsymbol{\alpha}} \in \R$
of the Newton interpolation can be computed
in closed form.
While Lagrange interpolation is primarily of theoretical
interest, the Newton form allows efficient and stable
evaluations of $Q_{G_{d,n}}f$ at any point $\mathrm{x} \in \square_d$. In particular, recent results in 
\cite{PIP1,PIP2,MIP,IEEE,minterpy} enable us to extend \eqref{eq:NEWT} for any choice of downward closed set $A\subseteq \N^d$, including the case of any $l_p$-degree e.g. total $l_1$-degree.

\begin{figure}[t!]
\begin{subfigure}{.45\textwidth}
  \centering\includegraphics[clip,width=1.0\columnwidth]{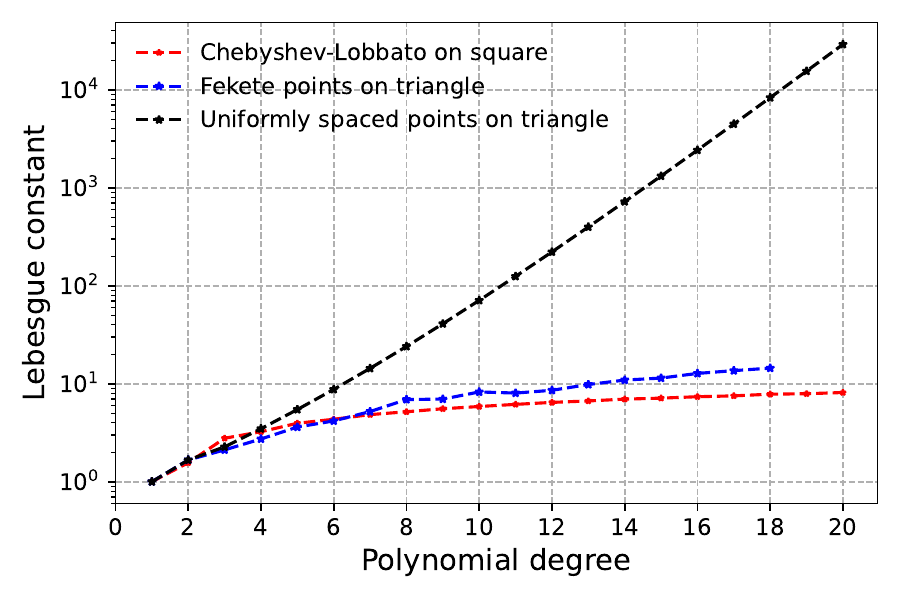}
      \caption{Lebesgue constant}
    \label{fig:T1}
\end{subfigure}%
\begin{subfigure}{0.27\textwidth}
    \includegraphics[width=\linewidth]{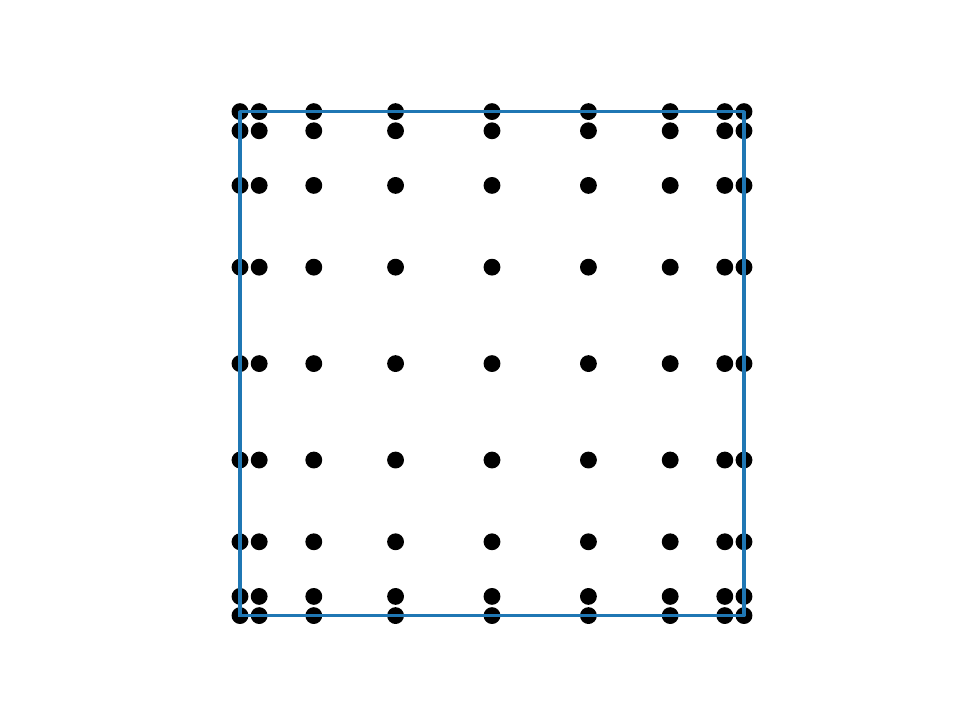}
    \caption{$\Cheb_{2,n}$}
      \label{fig:C1}
  \end{subfigure}%
  \begin{subfigure}{0.27\textwidth}
    \includegraphics[width=\linewidth]{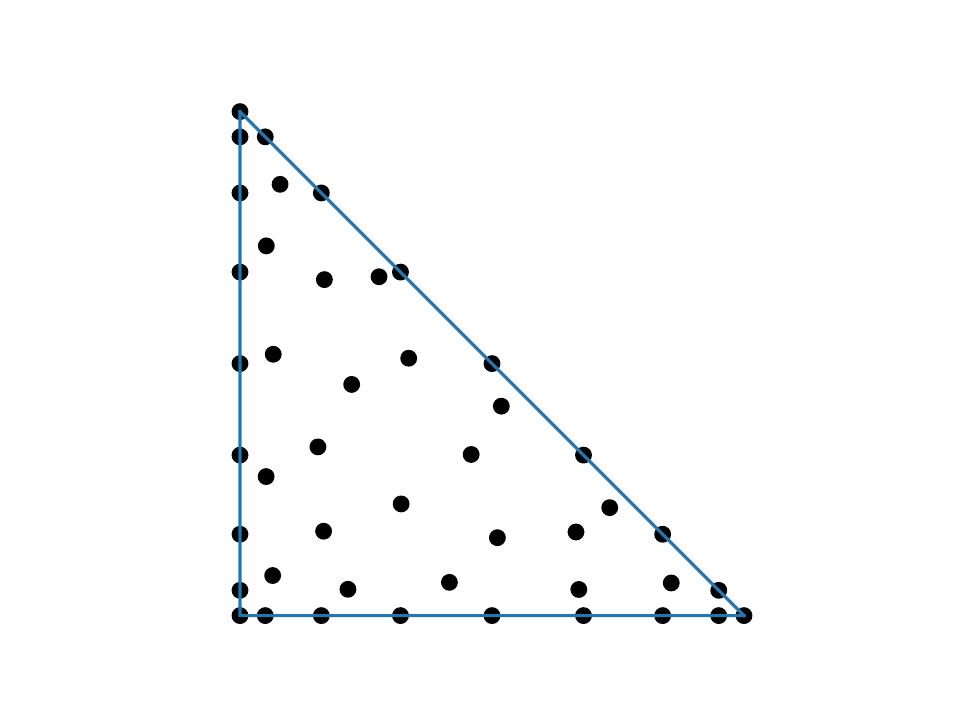}
    \caption{Fekete nodes}
    \label{feket_n=8}
  \end{subfigure}%

\caption{Lebesgue constants (\ref{fig:T1}) of uniformly spaced nodes on the triangle, Fekete nodes, and Chebyshev--Lobatto nodes (\ref{fig:C1}) a visualization of Chebyshev--Lobatto nodes and (\ref{feket_n=8}) Fekete nodes for $n=8$.
}
\label{fig:lebesgue_constant}
\end{figure}

The approximation power of polynomial interpolation is measured by
the \emph{Lebesgue constant}---the operator norm
of the interpolation operator
$Q_{G_{d,n}} : C^0(\square_d) \to \Pi_{d,n}$ given by 
\begin{equation}\label{eq:defLEB}
 \Lambda(G_{d,n})
 =
 \|Q_{G_{d,n}}\|
 = \sup_{ g\in C^0(\square_d)} \frac{\|Q_{G_{d,n}} g\|_{C^0(\square_d)}}{\|g\|_{C^0(\square_d)}}
             =  \Big\|\sum_{\boldsymbol{\alpha} \in A_{d,n}} |L_{\boldsymbol{\alpha}}|\Big\|_{C^0(\square_d)}\,.
\end{equation}
In the case of a one-dimensional interpolation domain $\square_1=[-1,1]$
and the Chebyshev--Lobatto grid
\begin{equation}\label{eq:CHEB}
   \Cheb_n = \li\{ \cos\Big(\frac{k\pi}{n}\Big) \; : \; 0 \leq k \leq n\re\}\,,
\end{equation}
the Lebesgue constant $\Lambda(\Cheb_n)$ increases slowly as $n \to \infty$.
Indeed,
\begin{equation}\label{eq:LEB}
 \Lambda(\Cheb_n)=\frac{2}{\pi}\big(\log(n+1) + \gamma +  \log(8/\pi)\big) + \Oc(1/n^2)\,,
\end{equation}
where $\gamma \approx  0.5772$ is the Euler--Mascheroni constant, see \cite{bernstein1931,ehlich,mccabe,rivlin,rivlin2,brutman}, surveyed by \cite{brutman2}, see also \cite{Trefethen2019}.
We extend this estimate to the $d$-dimensional case:
\begin{lemma}\label{lem:LEB}
The Lebesgue constant of the $d$-dimensional
Chebyshev--Lobatto grid
\begin{equation*}
 \Cheb_{d,n} = \bigoplus_{i=1}^d \Cheb_n
\end{equation*}
is $\Lambda(\Cheb_{d,n}) \leq  \Lambda(\Cheb_n)^d \in \Oc(\log(n+1)^d)
$.
\end{lemma}
\begin{proof} We consider the tensorial Lagrange polynomials $L_{\boldsymbol{\alpha}}(\mathrm{x}) = \prod_{i=1}^d l_{\alpha_i,i}(x_i)$ 
in the Chebyshev--Lobatto nodes,
with $l_{j,i}$ given in \eqref{eq:L} and obtain
\begin{align}
 \Lambda(\Cheb_{d,n})
 &=
  \Big\|\sum_{\boldsymbol{\alpha} \in A_{d,n}} |L_{\boldsymbol{\alpha}}| \Big\|_{C^0(\square_d)}    \leq \norm[\Big]{\sum_{\boldsymbol{\alpha} \in A_{d,n}} \prod_{i=1}^d |l_{\alpha_i,i}| }_{C^0(\square_d)} \label{eq:LEBB} \\
  &=
  \norm[\Big]{\big(\sum_{j=0}^n |l_{j,1}|\big) \cdots   \big(\sum_{j=0}^n |l_{j,l}|\big)\  \cdots \big(\sum_{j=0}^n |l_{j,d}|\big) }_{C^0(\square_d)}\,, \,\,\, 1<l<d\nonumber \\
 &\leq
 \prod_{i=1}^d \norm[\Big]{ \sum_{j=0}^n |l_{j,i}| }_{C^0(\square_d)}=  \prod_{i=1}^d \Lambda(\Cheb_n)\,. \nonumber
\end{align}
With \eqref{eq:LEB}, this yields $\Lambda(\Cheb_{d,n}) \leq \Lambda(\Cheb_n)^d \in \Oc(\log(n+1)^d)$.
\end{proof}

To demonstrate the advantage of interpolation
by tensor-product polynomial, we measured the
Lebesgue constants numerically by evaluating \eqref{eq:LEBB} on a very fine grid for two types
of interpolation: For  $l_\infty$-degree Chebyshev--Lobatto interpolation
on the square $\square_2$ and for total $l_1$-degree interpolation
in a uniform grid on the triangle $\triangle_2$. 
For  total $l_1$-degree interpolation in Fekete nodes on the triangle we use the Lebesgue constants from  \cite{Fekete11}. 

Fig.~\ref{fig:lebesgue_constant} shows the results.
We observe the Lebesgue constant of
 uniform triangle-grid  interpolation to 
rise quickly with increasing polynomial degree.
The Lebesgue constant for Chebyshev--Lobatto
interpolation increases much slower, while
the Lebesgue constant for Fekete nodes is only
marginal worse. 

However, Fekete nodes are only known up to degree $18$ \cite{Fekete11} in the
case of total $l_1$-degree interpolation and not for the  tensorial $l_\infty$-degree setting, which is a crucial ingredient of the approximation theory we deliver next.


\subsection{Approximation errors in terms of the \texorpdfstring{$r^\text{th}$}{r-th} total variation} 
We give a $d$-dimen- sional generalization of known error estimates with respect to the \emph{$r^{\text{th}}$ total variation}.
 We start  with a multivariate extension of a classic one-dimensional approximation result as presented in \cite{Trefethen2019},  building upon and extending the findings in \cite{caliari2008hyperinterpolation}.

 \begin{theorem}\label{theo:TRUNC}  Let $d\in\N$, $r\geq 0$, and $f$ be of bounded $r^{\text{th}}$ total variation, Definition~\ref{def:total_variation}. 
Then $f$ can be expanded in a Chebyshev series
\begin{equation*}
f(\mathrm{x})   =\sum_{\boldsymbol{\alpha}\in \N_0^d}c_{\boldsymbol{\alpha}} T_{\alpha_1}(x_1)\cdots T_{\alpha_d}(x_d)\,,
\end{equation*}
\begin{equation}\label{eq:bound}
    \text{with} \quad
|c_{\boldsymbol{\alpha}}|\leq V_{f,r}\Big(\frac{2}{ \pi q(q-1)\ldots(q-r)}\Big)^{d} \,,
\end{equation}
whenever $q =\min_{i=1,\dots,d} \alpha_i  \geq r+1$.
\end{theorem}
\begin{proof} We recall that the Chebyshev polynomials $T_{\boldsymbol{\alpha}}(\mathrm{x}) = \prod_{i=1}^dT_{\alpha_i}(x_i)$
are an orthonormal basis of $L^2(\square_d)$ with respect to the weighted
$L^2$ inner product with weight function $\omega_d(\mathrm{x}) =  \prod_{i=1}^d\frac{1}{\sqrt{1-x_i^2}}$. Due to \cite[Theorem~4.1]{MASON1980349}, any Lipschitz continuous function $f:\square_d \to \R$ has a uniformly and absolutely convergent multivariate Chebyshev series
with coefficients
\begin{align}
    c_{\boldsymbol{\alpha}} &= \frac{2^d}{\pi^d}\int_{\square_d}\omega(\mathrm{x})f(\mathrm{x})T_{\boldsymbol{\alpha}}(\mathrm{x})\,d\mathrm{x}    \label{eq:FUB}
\end{align}
for all $\boldsymbol{\alpha} \in \N_0^d$ with $\alpha_1,\dots,\alpha_d \ge 1$,
and with each factor $2/\pi$ replaced by $1/\pi$ in \eqref{eq:FUB} if $\alpha_i=0$ for some $1\leq i\leq d$.
 By following the argumentation in 1D, \cite[Theorem~7.1,~7.2]{Trefethen2019}, the coefficients are bounded by 
\begin{equation}
|c_{\boldsymbol{\alpha}}| \leq \Big(\frac{2}{ \pi q(q-1)\ldots(q-r)}\Big)^{d} \int_{\square_d} |\p^{\boldsymbol{\beta}} f (\mathrm{x})| d\mathrm{x} \,,
\end{equation}
where $\boldsymbol{\beta} = (r+1,\dots,r+1)$ and $q =\min_{i=1,\dots,d} \alpha_i  \geq r+1$. Consequently, by Definition~\ref{def:total_variation}, the estimate
\begin{equation*}
|c_{\boldsymbol{\alpha}}| \leq V_{f,r}\Big(\frac{2}{ \pi q(q-1)\ldots(q-r)}\Big)^{d}\,,
\end{equation*}
applies.
\end{proof}
We use this result in order to control the truncation error of the Chebyshev series.

\begin{corollary}\label{coir:trunc} Let the assumptions of Theorem~\ref{theo:TRUNC} be fulfilled. We denote with
\begin{equation}
\mathcal{T}_{f,n}(\mathrm{x})=\sum_{\boldsymbol{\alpha} \in A_{d,n}}c_{\boldsymbol{\alpha}}T_{\alpha_1}(x_1)\cdots T_{\alpha_d}(x_d)
\end{equation}
the truncated Chebyshev series of $f : \square_d \to \R$ with respect to $A_{d,n}$, with $n >r$.
\begin{enumerate}[label=\roman*)]
    \item \label{item:approx_1} The truncation error is bounded by
    \begin{equation}\label{eq:trunc}
        \|f - \mathcal{T}_{f,n}\|_{C^0(\square_d)} \leq \frac{2ed^2V_{f,r}}{\pi(r-d+1)}\left(\frac{n+1}{\,n+1-r\,}\right)^{r+1}\cdot \frac{1}{n^{r-d+1}} \in\Oc\big(n^{-(r-d+1)}\big)\,,
    \end{equation}
    $r>d-1$.
    \item \label{item:approx_2} The truncation 
    error of the first-order partial derivatives is bounded by
        \begin{equation}\label{eq:dtrunc}
        \|\p_{x_i} f - \p_{x_i} \mathcal{T}_{f,n}\|_{C^0(\square_d)} \leq \frac{2ed^2V_{f,r}}{\pi(r-d-1)}\left(\frac{n+1}{\,n+1-r\,}\right)^{r+1}\cdot\frac{1}{n^{r-d-1}}\in\Oc\big(n^{-(r-d-1)}\big)\,, 
    \end{equation}
    $r>d$, $\forall i=1,\dots,d.$
\end{enumerate}
\end{corollary}
\begin{proof}
\ref{item:approx_1} directly follows from Theorem~\ref{theo:TRUNC}:
Since \( T_k(\cos(x)) = \cos(kx) \) for all \( k \in \mathbb{N}_0 \), we observe that \( \|T_k\|_{C^0([-1,1])} \leq 1 \).
Let \(\Lambda_n^{(i)} = \{ \boldsymbol{\alpha} \in \mathbb{N}_{0}^d : \alpha_i > n \}\).
Then the truncation error admits the following bound:


\begin{align}
  \|f - \mathcal{T}_{f,n}\|_{C^0(\square_{d})}
  &\leq
  \sum_{\boldsymbol{\alpha} \in \mathbb{N}_{0}^d \setminus A_{d,n}} c_{\boldsymbol{\alpha}}\| T_{\alpha_1}(x_1)\cdots T_{\alpha_d}(x_d)\|_{C^0(\Omega_{d})}  \leq \sum_{\boldsymbol{\alpha} \in \mathbb{N}_{0}^d \setminus A_{d,n}} |c_{\boldsymbol{\alpha}}|\nonumber\\
  &\leq \sum_{i=1}^{d}  \sum_{\boldsymbol{\alpha} \in \Lambda_n^{(i)}}
\frac{2V_{f,r}}{ \pi \alpha_i(\alpha_i-1)\ldots(\alpha_i-r)}\leq \sum_{i=1}^{d}  \sum_{\boldsymbol{\alpha} \in \Lambda_n^{(i)}}
\frac{2V_{f,r}}{ \pi (\alpha_i-r)^{r+1}}\nonumber\\
  &= \frac{2dV_{f,r}}{\pi}\sum_{q=n+1}^{\infty} \frac{ |A_{d,q}\setminus A_{d,q-1}|}{(q - r)^{r+1}}\leq \frac{2ed^2V_{f,r}}{\pi}\sum_{q=n+1}^{\infty} \frac{ q^{d-1}}{(q - r)^{r+1}}\nonumber\\
  & \leq \frac{2ed^2V_{f,r}}{\pi}\left(\frac{n+1}{\,n+1-r\,}\right)^{r+1} \sum_{q=n+1}^{\infty} \frac{1}{q^{r-d+2}}\leq\frac{2ed^2V_{f,r}}{\pi}\left(\frac{n+1}{\,n+1-r\,}\right)^{r+1}\int_{n}^{\infty}\frac{1}{x^{r-d+2}}dx\nonumber\\\nonumber
  &= \frac{2ed^2V_{f,r}}{\pi(r-d+1)}\left(\frac{n+1}{\,n+1-r\,}\right)^{r+1}\cdot \frac{1}{n^{r-d+1}}\label{eq:sum_bound}\,
\end{align}
where for each \( \boldsymbol{\alpha} \in \Lambda_n^{(i)} \), we apply the Chebyshev coefficient bound~\eqref{eq:bound} in the \( i \)-th coordinate only, assuming \( q:=\alpha_i \geq r + 1 \). Moreover,we used the estimate \( |A_{d,q} \setminus A_{d,q-1}|
= (q+1)^d - q^d
\;\le\; e d q^{d-1}, \; (q > d)\), and bounded the resulting monotonically decreasing sum (for \(r-d+2 >1\)), which ultimately yields the desired estimate~\eqref{eq:trunc}.

We show \ref{item:approx_2} for the partial derivative $\p_{x_i}$ by writing
\begin{equation*}
    \norm{\partial_{x_i}f-\partial_{x_i}\mathcal{T}_{f,n}}_{C^0(\square_d)}\leq \sum_{\boldsymbol{\alpha} \in \N_0^d \setminus A_{d,n}} |c_{\boldsymbol{\alpha}}|\norm{T_{\alpha_1}\cdots T_{\alpha_{i-1}}}_{C^0(\square_d)}\norm{T'_{\alpha_i}}_{C^0(\square_d)}\norm{T_{\alpha_{i+1}}\cdots T_{\alpha_d}}_{C^0(\square_d)}\,.
\end{equation*}
We recall that $T_{k}(x)=\cos{\left(k\arccos(x\right))}$ for $-1\leq x\leq 1$ and deduce
that for all $k \in \N_0$
\begin{equation}\label{eq.derivative}
T'_{k}(x)=\frac{k\sin{\left(k\arccos\left(x\right)\right)}}{\sqrt{1-x^{2}}} =\frac{k\sin{\left(kt\right)}}{\sin{\left(t\right)}}\,, \qquad t=\arccos(x)\,,
\end{equation}
yielding $\|T'_{\alpha_i}\|_{C^0(\square_d)}=\alpha_i^{2}$.
Following~\ref{item:approx_1}, we compute

\begin{align}
    \norm{\partial_{x_i}f-\partial_{x_i}\mathcal{T}_{f,n}}_{C^0(\square_d)}
    &\leq \sum_{\boldsymbol{\alpha} \in \N_0^d \setminus A_{d,n}} |c_{\boldsymbol{\alpha}}|\alpha_i^2  \leq \sum_{i=1}^{d}  \sum_{\boldsymbol{\alpha} \in \Lambda_n^{(i)}}
 \frac{2V_{f,r}}{ \pi (\alpha_i-r)^{r+1}}\alpha_i^2\nonumber \\ 
&\leq \frac{2ed^2V_{f,r}}{\pi}\sum_{q=n+1}^{\infty} \frac{q^{d-1}q^2}{(q-r)^{r+1}} \leq \frac{2ed^2V_{f,r}}{\pi}\left(\frac{n+1}{\,n+1-r\,}\right)^{r+1} \sum_{q=n+1}^{\infty} \frac{1}{q^{r-d}}\nonumber\\
&\leq \frac{2ed^2V_{f,r}}{\pi}\left(\frac{n+1}{\,n+1-r\,}\right)^{r+1}\int_{n}^{\infty}\frac{1}{x^{r-d}}=\frac{2ed^2V_{f,r}}{\pi(r-d-1)}\left(\frac{n+1}{\,n+1-r\,}\right)^{r+1}\cdot\frac{1}{n^{r-d-1}} \,,\label{eq:EST}
\end{align}
by bounding the monotonically decreasing sum (for  \(r-d >1\))  in \eqref{eq:EST}.
\end{proof}

With the previous results, we can bound the
approximation error of the Chebyshev--Lobatto interpolant of $f$. 

\begin{corollary}\label{corollary.ineq}
Let the assumption of Theorem~\ref{theo:TRUNC} be satisfied and 
$Q_{G_{d,n}}f$ be the interpolant of $f : \square_d \to \R$ 
in the Chebyshev--Lobatto grid $\Cheb_{d,n}$.
Then the approximation errors of $f$ and its first derivatives are bounded by
\begin{align}\label{eq:app}
    \|f-Q_{G_{d,n}}f\|_{C^0(\square_d)}
    & \leq \frac{4ed^2V_{f,r}}{\pi(r-d+1)}\left(\frac{n+1}{\,n+1-r\,}\right)^{r+1}\cdot \frac{1}{n^{r-d+1}}\in\Oc\big(n^{-(r-d+1)}\big)  \,.
\intertext{and}
 \label{eq:dapp}
    \|\p_{x_i} f-\p_{x_i} Q_{G_{d,n}}f\|_{C^0(\square_d)}
    & \leq \frac{4ed^2V_{f,r}}{\pi(r-d-1)}\left(\frac{n+1}{\,n+1-r\,}\right)^{r+1}\cdot\frac{1}{n^{r-d-1}}\in\Oc\big(n^{-(r-d-1)}\big)\,
\end{align}
for all $i =1,\dots,d$.
\end{corollary}
\begin{proof} The statement is a direct consequence of Theorem~\ref{theo:TRUNC} and the [Aliasing Theorems 4.1, 4.2] \cite{Trefethen2019}, stating that 
\begin{equation}
f(\mathrm{x})-Q_{G_{d,n}}f(\mathrm{x}) = \sum_{\boldsymbol{\alpha} \in \N_0^d \setminus A_{d,n}}c_{\boldsymbol{\alpha}}\big( T_{\alpha_1}(x_1)\cdots T_{\alpha_d}(x_d) - T_{\beta_1}(x_1)\cdots T_{\beta_d}(x_d)\big)\,, 
\end{equation}
where $\beta_i = |(\alpha_i + n-1)\mathrm{mod} 2n - (n-1)|$. This shows that, when following the estimation in Corollary~\ref{coir:trunc}, the approximation error of the interpolant can be bounded by twice the bound, appearing for the truncation.
\end{proof}

\begin{remark}[Exponential approximation rates] 
If the total variation $V_{f,r}$ is uniformly
bounded in $r$, i.e.,
$\limsup_{r \to \infty} V_{f,r}  < \infty$,
Corollar~\ref{corollary.ineq} implies that for $n> r\in \N$ large enough 
\begin{equation}
\frac{2ed^2V_{f,r}}{\pi(r-d+1)}\left(\frac{n+1}{\,n+1-r\,}\right)^{r+1}\cdot \frac{1}{n^{r-d+1}} \leq  CR^{-n}\,, \quad \text {for some} \,\, 1 < R \,, C\in \R^+\,.
\end{equation}
Hence, the error bounds
\eqref{eq:app} and \eqref{eq:dapp} imply
exponential error decay for increasing degree $n \in \N$.
\end{remark}

\section{Integration errors of high-order volume elements (HOVE) }\label{sec:APPP_HOS}

We derive the integration error for replacing  the surface geometry $\varphi_i$ and
the integrand $f$ by Chebyshev--Lobatto interpolants
 $Q_{G_{d,k}} \varphi_i$, $Q_{G_{d,n}}(f \circ \varphi_i)$, respectively. As we show in Corollary~\ref{cor:pull}, the resulting closed form expression of 
the integral can be computed precisely by high-order quadrature rules:
\begin{align}
\int_S f\,dS 
&\approx
\sum_{i=1}^K\int_{\square_d} Q_{G_{d,n}}(f\circ\varphi_i)(\mathrm{x}) \sqrt{\det((DQ_{G_{d,k}}\varphi_i(\mathrm{x}))^T DQ_{G_{d,k}}\varphi_i(\mathrm{x}))}\,d\mathrm{x} \nonumber \\ 
&\approx
\sum_{i=1}^K \sum_{\mathrm{p} \in P}\omega_{\mathrm{p}} \; Q_{G_{d,n}}(f \circ\varphi_i)(\mathrm{p})
\sqrt{\det((DQ_{G_{d,k}}\varphi_i(\mathrm{p}))^T DQ_{G_{d,k}}\varphi_i(\mathrm{p}))}\,. \label{eq:HOSQ}
\end{align}


We start by bounding the approximation error of the geometry.

\begin{lemma}\label{lem:Jac}
Let $S$ be a $d$-dimensional $C^{r+1}$-surface, $r \geq0$,
and $\varphi_i = \varrho_i \circ \sigma : \square_d \to \R^m$, $i=1,\dots, K$ be 
a $r$-regular cubical re-parametrization, Definition~\ref{def:cubical_reparametrization}. 
Let $Q_{G_{d,k}} \varphi_i$ be
the vector-valued tensor-polynomial interpolant
of $\varphi_i$ in the Chebyshev--Lobbatto grid $\Cheb_{d,k}$.
\begin{enumerate}[label={\roman*})]
    \item \label{item:difference_jacobians} The Jacobians of $\varphi_i$ and its interpolant $Q_{G_{d,k}}\varphi_i$ differ by
    \begin{equation}\label{eq:Jac}
        \|D\varphi_i - DQ_{G_{d,k}}\varphi_i\|_{C^0(\square_d)}
         \leq \frac{4ed^2V_{\varphi,r}}{\pi(r-d-1)}\left(\frac{k+1}{\,k+1-r\,}\right)^{r+1}\cdot \frac{1}{k^{r-d-1}}\,,
    \end{equation}

   where $V_{\varphi_i,r}$ is the maximum $r^{\text{th}}$ total variation of the coordinate functions of $\varphi_i$.

    \item \label{item:volume_element_difference} The difference of the volume elements is bounded by
\begin{equation*}
\bigl\|\sqrt{\det(\Phi_i)}-\sqrt{\det(\Psi_i)}\bigr\|_{C^0(\square^d)}
\;\leq\; d\,C^{\,d-1}\,
\bigl\|D\varphi_i - DQ_{G^d,k}\varphi_i\bigr\|_{C^0(\square^d)}\;,
\end{equation*}
    where $\Phi_i = D\varphi_i^TD\varphi_i$, $\Psi_i = DQ_{G_{d,k}}\varphi_i^TDQ_{G_{d,k}}\varphi_i$.
\end{enumerate}
\end{lemma}
\begin{proof}
\ref{item:difference_jacobians} follows directly from  Corollary~\ref{corollary.ineq}, \eqref{eq:dapp},
whereas \ref{item:volume_element_difference} can be estimated 
in terms of singular values. For 
\(\Phi_i = D\varphi_i^{\!\top} D\varphi_i\) and 
\(\Psi_i = (DQ_{G_{d,k}}\varphi_i)^{\!\top}DQ_{G_{d,k}}\varphi_i\), 
let \(\sigma_1(\cdot)\geq\cdots\geq\sigma_d(\cdot)\) denote the singular values. Then
\begin{equation}
\sqrt{\det(\Phi_i)} = \prod_{j=1}^d \sigma_j(D\varphi_i), 
\qquad 
\sqrt{\det(\Psi_i)} = \prod_{j=1}^d \sigma_j(DQ_{G_{d,k}}\varphi_i).
\end{equation}
By the mean-value inequality for products,
\begin{equation}
\bigl|\sqrt{\det(\Phi_i)} - \sqrt{\det(\Psi_i)}\bigr|
\;\leq\; \sum_{k=1}^d 
|\sigma_k(D\varphi_i)-\sigma_k(DQ_{G_{d,k}}\varphi_i)|\,
\prod_{j\neq k}\max\{\sigma_j(D\varphi_i),\sigma_j(DQ_{G_{d,k}}\varphi_i)\}.
\end{equation}
Now $\sigma_j(D\varphi_i \leq \|D\varphi_i \|$ and  $k\gg 1$ large enough 
\begin{align*}
\sigma_j(DQ_{G_{d,k}}\varphi_i) 
&\leq \|DQ_{G_{d,k}}\varphi_i\|_{C^0(\square^d)} \\
&\leq \|D\varphi_i\|_{C^0(\square^d)} 
   + \|DQ_{G_{d,k}}\varphi_i - D\varphi_i\|_{C^0(\square^d)} \\
&\leq \|D\varphi_i\|_{C^0(\square^d)} + 1.
\end{align*}
We set 
$C := \max_{i=1,\dots,K}\|D\varphi_i\|_{C^0(\square^d)} + 1$
and obtain 
\begin{align*}
\big|\sqrt{\det(\Phi_i)} - \sqrt{\det(\Psi_i)}\big|
&\leq  \sum_{k=1}^d 
|\sigma_k(D\varphi_i)-\sigma_k(DQ_{G_{d,k}}\varphi_i)|\, C^{d-1} \\
&\leq  d\,C^{d-1}\,
\|D\varphi_i - DQ_{G_{d,k}}\varphi_i\|_{C^0(\square^d)},
\end{align*}
where we used Weyl’s inequality for the second line.

\end{proof}

With the help of the previous result, we bound the integration error.
\begin{theorem}[Integration error] \label{Main.Thm}
Let the  assumptions of Lemma~\ref{lem:Jac} be satisfied, and let $f:S \to \R$ be of bounded $r^{\text{th}}$ total variation $V_{f,r}$.
For each mesh element, we consider its approximation $Q_{G_{d,n}}$ by
tensor-polynomial interpolation in the Chebyshev--Lobbatto grid $\Cheb_{d,n}$.
Then the integration error induced by the approximation of the geometry $\varphi_i$ and of $f\circ \varphi_i$ is
\begin{multline*}
\Big|
 \int_S f\,dS - \sum_{i=1}^K \int_{\square_d} Q_{G_{d,n}}(f\circ \varphi_i)(\mathrm{x}) \sqrt{\det\big((DQ_{G_{d,k}}\varphi_i(\mathrm{x}))^T DQ_{G_{d,k}}\varphi_i(\mathrm{x})\big)}\,d\mathrm{x} \Big|\\
  \leq
  \ee_f\vol(S) + \ee_f\ee_\varphi \vol(\square_d) + \|f\|_{C^0(S)}\ee_\varphi\,\vol(\square_d)
  =
 \Oc\Big(\frac{1}{n^{r-d+1}}\Big)+\Oc\Big(\frac{1}{k^{r-d-1}}\Big)\,,
\end{multline*}
where $\vol(S)$ and $\vol(\square_d)$ denote the volumes of $S$ and $\square_d$, respectively, and
\begin{align*}
\ee_f = \frac{4ed^2V_{f,r}}{\pi(r-d+1)}\left(\frac{n+1}{\,n+1-r\,}\right)^{r+1}\cdot\frac{1}{n^{r-d+1}}\,, \quad
&
\mu_\varphi= \frac{4ed^2V_{\varphi,r}}{\pi(r-d-1)}\left(\frac{k+1}{\,k+1-r\,}\right)^{r+1}\cdot \frac{1}{k^{r-d-1}}\,,
\end{align*}
with
\begin{equation*}
\ee_\varphi \;=\; d\,C^{\,d-1}\,\mu_\varphi,
\qquad
C:= \max_{i=1,\dots,K}\|D\varphi_i\|_{C^0(\square_d)}+1\,
\qquad
V_{\varphi,r}=\max_{i=1,\dots,K} V_{\varphi_i,r}\,.
\end{equation*}
\end{theorem}
\begin{proof} We  set  $\Phi_i = (D\varphi_i^TD\varphi_i)^{1/2}$, $\Psi_i = (DQ_{G_{d,k}}\varphi_i^TDQ_{G_{d,k}}\varphi_i)^{1/2}$, apply the Cauchy-- \linebreak Schwarz inequality, and estimate
\begin{multline*}
\Big|
 \int_S f\,dS 
 -
 \sum_{i=1}^K \int_{\square_d} Q_{G_{d,n}}(f\circ \varphi_i)(\mathrm{x}) \sqrt{\det(\Phi_i(\mathrm{x}))}\,d\mathrm{x} \Big| \\
 \begin{aligned}
 &\leq
     \sum_{i=1}^K \int_{\square_d} \big|f(\varphi_i(\mathrm{x}))  -Q_{G_{d,n}}(f\circ \varphi_i)(\mathrm{x})\big|\sqrt{\det( \Psi_i(\mathrm{x}))}\,d\mathrm{x} \\
     &+ \sum_{i=1}^K \int_{\square_d} \big|f(\varphi_i(\mathrm{x}))  -Q_{G_{d,n}}(f\circ \varphi_i)(\mathrm{x})\big| \cdot \norm[\big]{\sqrt{\det(\Phi_i(\mathrm{x}))} - \sqrt{\det(\Psi_i(\mathrm{x}))}}_{C^0(\square_d)} \,d\mathrm{x}
     \\
     &+ \sum_{i=1}^K \int_{\square_d} \big|f(\varphi_i(\mathrm{x})) \big| \cdot \norm[\big]{\sqrt{\det(\Phi_i(\mathrm{x}))} - \sqrt{\det(\Psi_i(\mathrm{x}))}}_{C^0(\square_d)} \,d\mathrm{x}\\
     &\leq \ee_f\vol(S) + \ee_f\ee_\varphi  \vol(\square_d) + \|f\|_{C^0(S)}\ee_\varphi\vol(\square_d)\,.
     \end{aligned}
\end{multline*}
    The estimates for $\ee_f$, $\ee_\varphi$, $\mu_{\varphi}$ follow from
Corollary~\ref{corollary.ineq} and
Lemma~\ref{lem:Jac}, concluding the proof.
\end{proof}

The approximated integral can now be computed using a
quadrature rule. There are two basic options: Either use
a quadrature rule for the cube domain $\square_d$ directly,
or use a simplex rule and pull it back to $\square_d$ by the
inverse of the square-squeezing map $\sigma_*$
(effectively integrating over the original triangulation
$\{\varrho_i\}$ of $S$ from Definition~\ref{def:triangulation}).
While the former seems more natural, the latter is
more efficient, as simplex rules typically consist of
fewer nodes.

\begin{corollary}[Quadrature rule error]
\label{cor:pull}
Under the assumptions of Theorem~\ref{Main.Thm}  denote with $\Phi_i = D\varphi_i^TD\varphi_i$, $\Psi_i = DQ_{G_{d,k}}\varphi_i^TDQ_{G_{d,k}}\varphi_i$, and 
$Q_{G_{d,n}}f \in \Pi_{d,n}$  the polynomial approximations of $f$ and $\varphi_i$ of $l_\infty$-degree $n,k \in \N$. Then there is $0<v<1$, independent of $l$, such that for $1 \leq l \in \N$ large enough: 

\begin{enumerate}[left=0pt,label=\roman*)]
    \item\label{Q1} Let  
$\mathrm{p} \in P,\omega_\mathrm{p}$ be the nodes and weights of the tensorial Gauss--Legendre quadrature on $\square_d$ \cite{stroud} of order $N\in \N$, integrating any polynomial $Q \in \Pi_{d,M}$ of $l_\infty$-degree  $M =2kdl + n$ exactly. Then
\begin{multline}
\int_{\square_d} Q_{G_{d,n}}f(\mathrm{x}) \sqrt{\det\big((DQ_k\varphi_i(\mathrm{x}))^T DQ_k\varphi_i(\mathrm{x})\big)}\,d\mathrm{x} \\
=  \sum_{\mathrm{p} \in P}\omega_{\mathrm{p}} 
 Q_{G_{d,n}}f(\mathrm{p})\sqrt{\det\big((DQ_k\varphi_i(\mathrm{p}))^TDQ_k\varphi_i(\mathrm{p})\big)} + \Oc(v^{l+1})\,.
\end{multline}

\item\label{Q2} Let $\sigma: \square_d \to \triangle_d$ be a cube--simplex transformation diffeomorphic in the interior $\mathring{\square}_d$,  $P^* = \{\mathrm{p^*} =\sigma^{-1}(\mathrm{q}): \mathrm{q} \in P \subseteq \mathring{\triangle_d}\}$,
$\omega_\mathrm{p^*} = \omega_{\mathrm{q}} \sqrt{\det((D\sigma^{-1}(\mathrm{q}))^TD\sigma^{-1}(\mathrm{q}))}$
be the $\sigma$-pull-back rule of a simplex rule of order $N^*\in \N$, integrating any polynomial $Q \in \Pi_{d,M}$ of $l_\infty$-degree 
$M =2kdl + n$ exactly on $\triangle_d$.  Then
\begin{multline}
\int_{\triangle_d} Q_{G_{d,n}}f(\sigma^{-1}(\mathrm{y})) \sqrt{\det\big((D\varrho_i(\mathrm{y}))^T D\varrho_i(\mathrm{y})\big)}d\mathrm{y} \\
 =    \sum_{\mathrm{p^*} \in P^*}\omega_{\mathrm{p^*}} 
 Q_{G_{d,n}}f(\mathrm{p^*})\sqrt{\det\big((DQ_k\varphi_i(\mathrm{p^*}))^T DQ_k\varphi_i(\mathrm{p^*})\big)} + \Oc(\varepsilon_iI_{i,\triangle_d}) + \Oc(v^{l+1})\,,
\end{multline}
with  $\varepsilon_i = \|\sqrt{\det(\Phi_i)} - \sqrt{\det(\Psi_i)}\|_{C^0(\square_d)}$ as in Lemma~\ref{lem:Jac} and $I_{i,\square_d}= \int_{\square_d} Q_{G_{d,n}}f(\mathrm{x}) d\mathrm{x}$.
\end{enumerate}

\end{corollary}
\begin{proof}
To prove \ref{Q1}, we choose $\kappa >\|\sqrt{\det(\Psi_i)}\|_{C^0(\square_d)}$ and rewrite: 
\begin{equation*}
\kappa\sqrt{\det\big(\frac{1}{\kappa^2}\Psi_i(\mathrm{x})\big)} = \kappa\sqrt{1+x}\,, \,\,x =\frac{1}{\kappa^2}\det(\Psi_i(\mathrm{x})) -1\,.
\end{equation*}
We recall that $\sqrt{1 + x} = \sum_{s=0}^\infty\frac{(-1)^s2s!}{(1-2s)(s!)^2(4^s)}x^s$, for $|x|<1$, and deduce that
\begin{equation*}
\int_{\square_d} Q_{G_{d,n}}f(\mathrm{x}) \kappa\sqrt{\det(\frac{1}{\kappa^2}\Psi_i(\mathrm{x}))}\,d\mathrm{x} =  \int_{\square_d}Q_{G_{d,n}}f(\mathrm{x}) Q(\mathrm{x}) d\mathrm{x} + \Oc(v^{l+1})\,,
\end{equation*}
where $Q$ has $l_\infty$-degree $M -n$. Hence,   $\int_{\square_d} Q_{G_{d,n}}f(\mathrm{x})  Q(\mathrm{x}) d\mathrm{x}=  \sum_{\mathrm{p} \in P}\omega_{\mathrm{p}} 
 Q_{G_{d,n}}f(\mathrm{p}) Q_{l}(\mathrm{p})$ can be computed exactly due to the Gauss--Legendre quadrature of order $N$. 
Consequently, \ref{Q1} is proven. Now \ref{Q2} follows from \ref{Q1} by 
\begin{multline*}
\int_{\triangle_d} Q_{G_{d,n}}f(\sigma^{-1}(\mathrm{y})) \sqrt{\det\big((D\varrho_i(\mathrm{y}))^T D\varrho_i(\mathrm{y})\big)}\,d\mathrm{y} \\
\begin{aligned}
&= \int_{\triangle_d} Q_{G_{d,n}}f(\sigma^{-1}(\mathrm{y})) \sqrt{\det\big((D\sigma^{-1}(\mathrm{y}))^T (D\varphi_i(\sigma^{-1}(\mathrm{y})))^T D\varphi_i(\sigma^{-1}(\mathrm{y}))D\sigma^{-1}(\mathrm{y})\big)}\, d\mathrm{y}\\ 
& = 
\int_{\triangle_d} Q_{G_{d,n}}f(\sigma^{-1}(\mathrm{y})) \sqrt{\det(\Psi(\sigma^{-1}(\mathrm{y})))\det(D\sigma^{-1}(\mathrm{y})^TD\sigma^{-1}(\mathrm{y}))}\, d\mathrm{y}\\
& + \int_{\square_d} Q_{G_{d,n}}f(\mathrm{x}) \Big(\sqrt{\det(\Phi_i(\mathrm{x}))} - \sqrt{\det(\Psi_i(\mathrm{x}))}\Big) d\mathrm{x}\\
&= \sum_{\mathrm{p^*} \in P^*}\omega_{\mathrm{p^*}} 
 Q_{G_{d,n}}f(\mathrm{p^*})\sqrt{\det(\Psi_i(p^*))} + \Oc(\varepsilon_iI_{i,\triangle_d}) 
+ \Oc(v^{l+1})\,,
 \end{aligned}
\end{multline*}
proving the statement.
\end{proof}

\begin{remark} In fact, Corollary~\ref{cor:pull}~\ref{Q2} applies for the square-squeezing transformation $\sigma_*$ and Duffy transformation $\sigma_{\text{Duffy}}$ in combination with the symmetric Gauss quadrature $\mathrm{q} \in P,\omega_\mathrm{q}$ of the triangle $\triangle_2$ \cite{dunavant1985high} (both are diffeomorphisms in the interior $\mathring{\square}_2$ and $P\subseteq  \mathring{\triangle}_2$).
\end{remark}
While Corollary~\ref{cor:pull} suggests the necessity of a high order  quadrature, $M \gg k,n$, as part of the next section, we empirically find that choosing $M=k=n$ equally to the interpolation degrees suffices for achieving computations reaching machine precision.  
 

\section{Numerical experiments}\label{sec:NUM}

We now demonstrate the quality of the HOVE surface integration
method described in Section~\ref{sec:APPP_HOS}, by presenting several numerical experiments.
We focus on the important case of two-dimensional manifolds
exclusively. We triangulate these manifolds by first
approximating them by piecewise affine triangulations
in $\R^3$, constructed by the algorithm of Persson and Strang~\cite{Persson}.
The flat triangles are then equipped
with the Euclidean closest-point
projections, approximating the maps $\pi_i : T_i \to S$
as described in Remark~\ref{rem:affine_mesh_with_projections}.


We compare HOVE with the \textsc{Dune-CurvedGrid} integration algorithm (DCG),
included in the surface-parametrization module \texttt{dune-curvedgrid} \cite{CurvedGrid} of the \textsc{Dune} finite element framework.%
\footnote{\url{www.dune-project.org}}
As  discussed in \cite{Zavalani23}, DCG interpolates the closest-point projection directly on each triangle,
using $l_1$-degree polynomials on a uniform point set. 

If not stated otherwise, HOVE uses square-squeezing pull-backs of symmetric
Gauss triangle rules \cite{dunavant1985high} as quadratures on $\square_2$ (Corollary~\ref{cor:pull}~\ref{Q2}). Similarly, for DCG, we also make use of symmetric Gauss triangle rules of the same degree as used in HOVE.

Our implementation of HOVE is part of a Python package called \textsc{surfgeopy}.%
\footnote{\url{https://github.com/casus/surfgeopy}
}
The examples and results of this manuscript using \textsc{Dune-CurvedGrid}  
are summarized and made available in a separate repository.%
%
\footnote{\url{https://github.com/casus/dune-surface_int}
}

\subsection{Duffy-transform-integration vs square-squeezing-integration  }
This first experiment investigates the impact of interpolaing the volume element of the sphere with $1^{\text{st}}$ or $2^{\text{nd}}$ kind Chebyshev nodes in conjunction with the Duffy and square-squeezing transformations, respectively, and posterior computing the area of one octant of the unit sphere. Hereby,  Fejér’s rule is applied for $1^{\text{st}}$ kind Chebyshev nodes, while the Clenshaw-Curtis quadrature is employed for $2^{\text{nd}}$ kind Chebyshev nodes, each of order equal to the interpolation degree.  
\begin{figure}[!t]
\centering
\begin{subfigure}[t]{0.48\textwidth}  
    \centering
    \includegraphics[width=\textwidth]{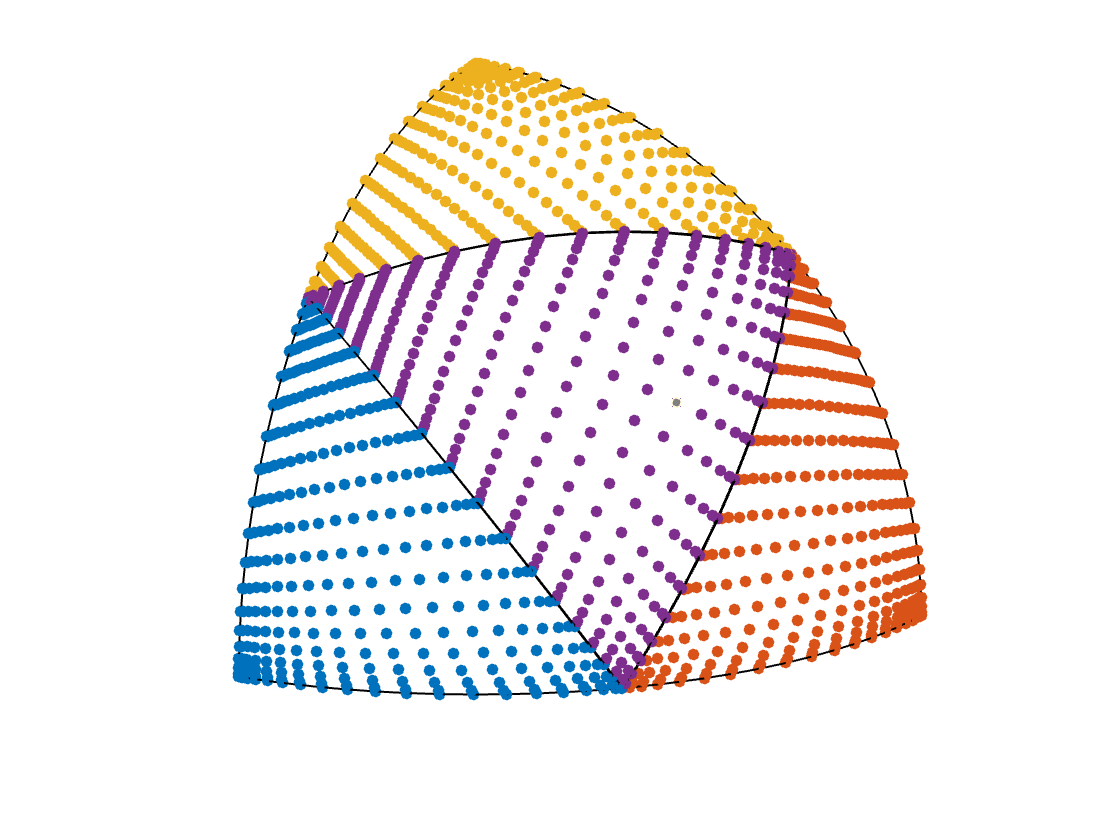}
    \caption{Chebyshev nodes of the $1^{\text{st}}$ mapped by the Duffy transformation.}
    \label{fig:first_kind}
\end{subfigure}%
\hfill
\begin{subfigure}[t]{0.48\textwidth}  
    \centering
    \includegraphics[width=\textwidth]{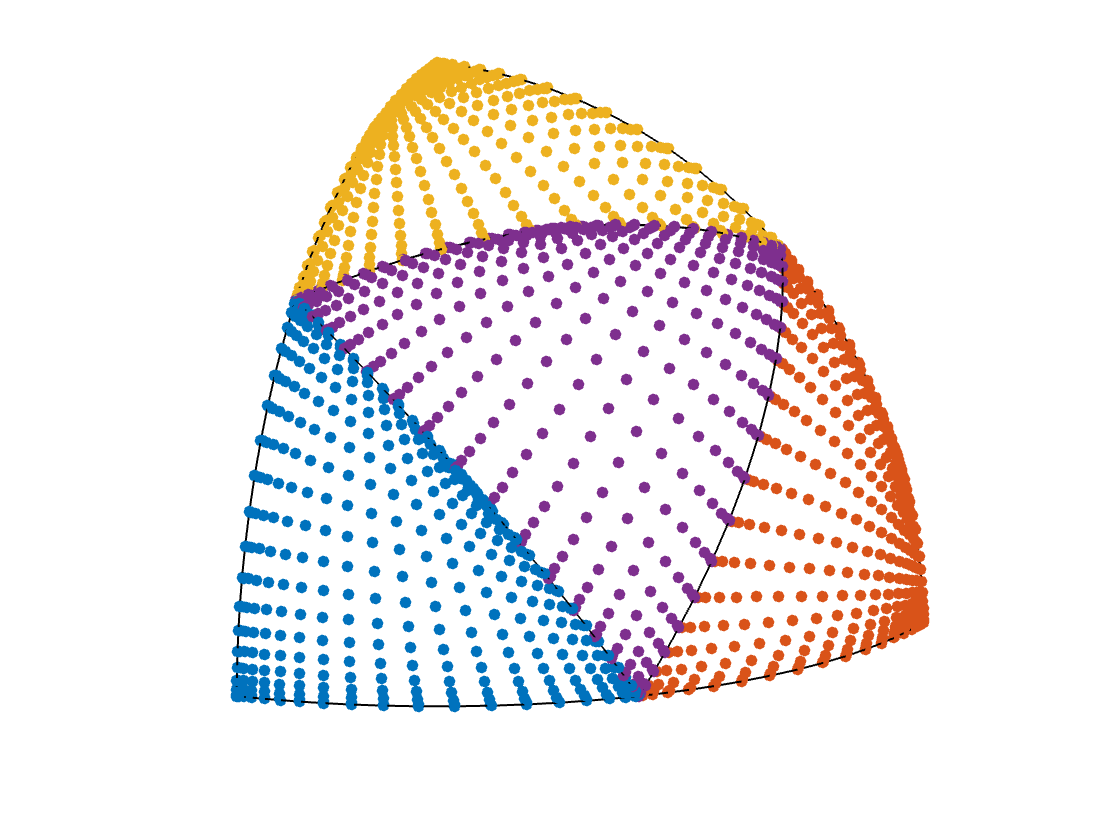}
    \caption{Chebyshev nodes of the $2^{\text{nd}}$ mapped by square-squeezing.}
    \label{fig:second_kind}
\end{subfigure}


\begin{subfigure}[b]{0.48\textwidth}  
    \centering
    \includegraphics[width=\textwidth]{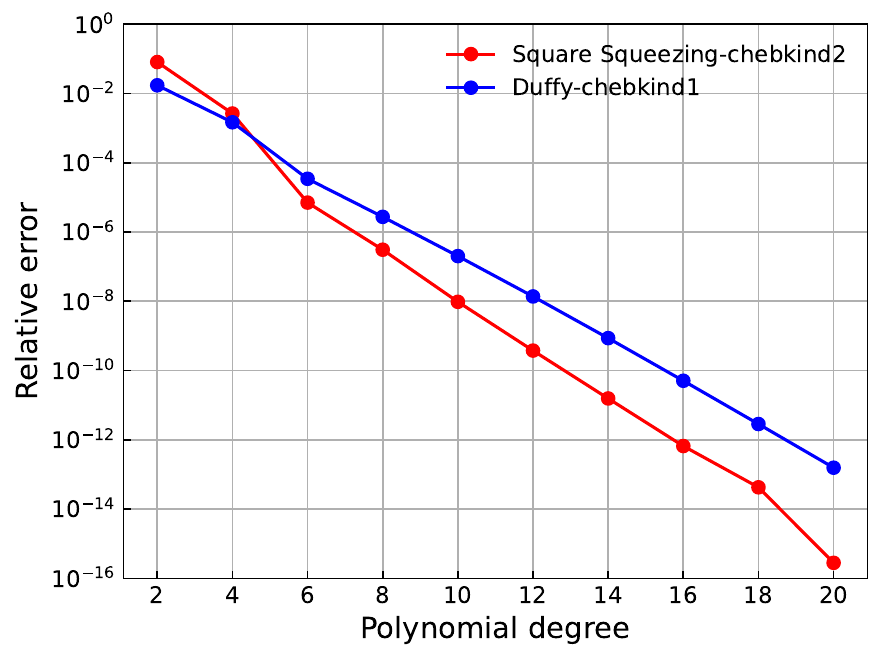}
    \caption{Square-squeezing-integration with $2$nd kind Chebyshev nodes vs Duffy with $1$st kind Chebyshev nodes}
    \label{fig:rel_second_kind}
\end{subfigure}%
\hfill
\begin{subfigure}[b]{0.48\textwidth}  
    \centering
    \includegraphics[width=\textwidth]{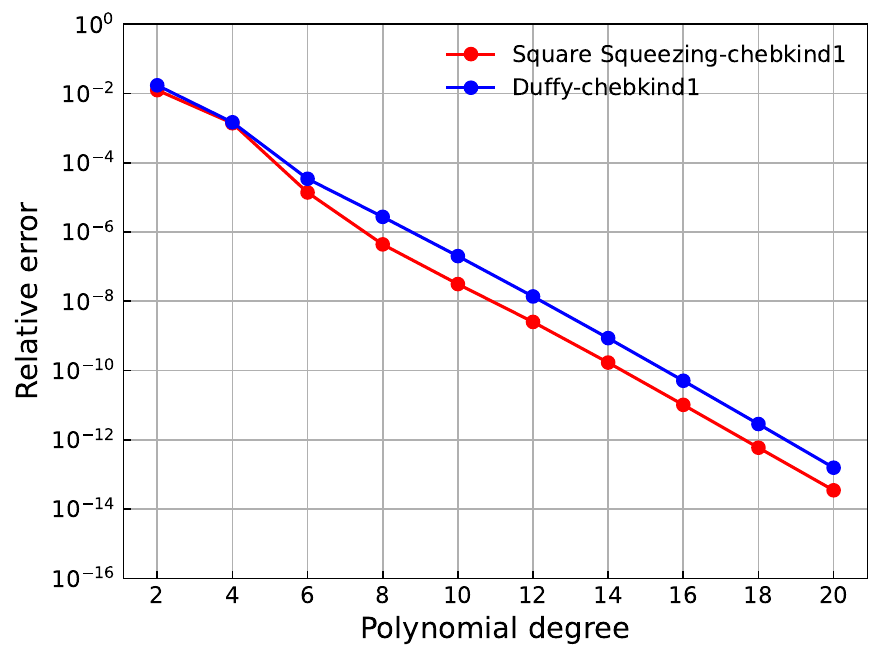}
    \caption{Square-squeezing-integration vs Duffy-transform-integration with $1$st kind Chebyshev nodes for both.}
    \label{fig:ss_duffy_cheb1}
\end{subfigure}

\caption{Chebyshev nodes mapped onto a triangulation of one octant of the unit sphere (\ref{fig:first_kind}), (\ref{fig:second_kind}), along with the relative error of square-squeezing-integration (\ref{fig:rel_second_kind}) and Duffy-transform-integration (\ref{fig:ss_duffy_cheb1}).}
\end{figure}

Fig.~\ref{fig:rel_second_kind} shows the appearing relative errors. In both cases we observe an exponential error decay. However, square-squeezing-integration achieves two orders of magnitude higher accuracy. Specifically, for $\deg = 20$ 
Duffy-transform-integration results in an error of $2.4736\times 10^{-13}$, while  square-squeezing-integration  achieves $4.4409 \times 10^{-16}$. 
Additionally, comparing both transformations, relying on $1$st kind Chebyshev nodes, Fig.~\ref{fig:ss_duffy_cheb1}., still shows and advantage of exploiting square-squeezing instead of the Duffy transformation.

Given that significant enhancement in accuracy performance already for this simple integration task suggests a high impact of the HOVE approach, being further investigated below.

\subsection{Surface area}

\begin{figure}[t!]
\begin{subfigure}{.5\textwidth}
  \centering
  \includegraphics[clip,width=1\columnwidth]{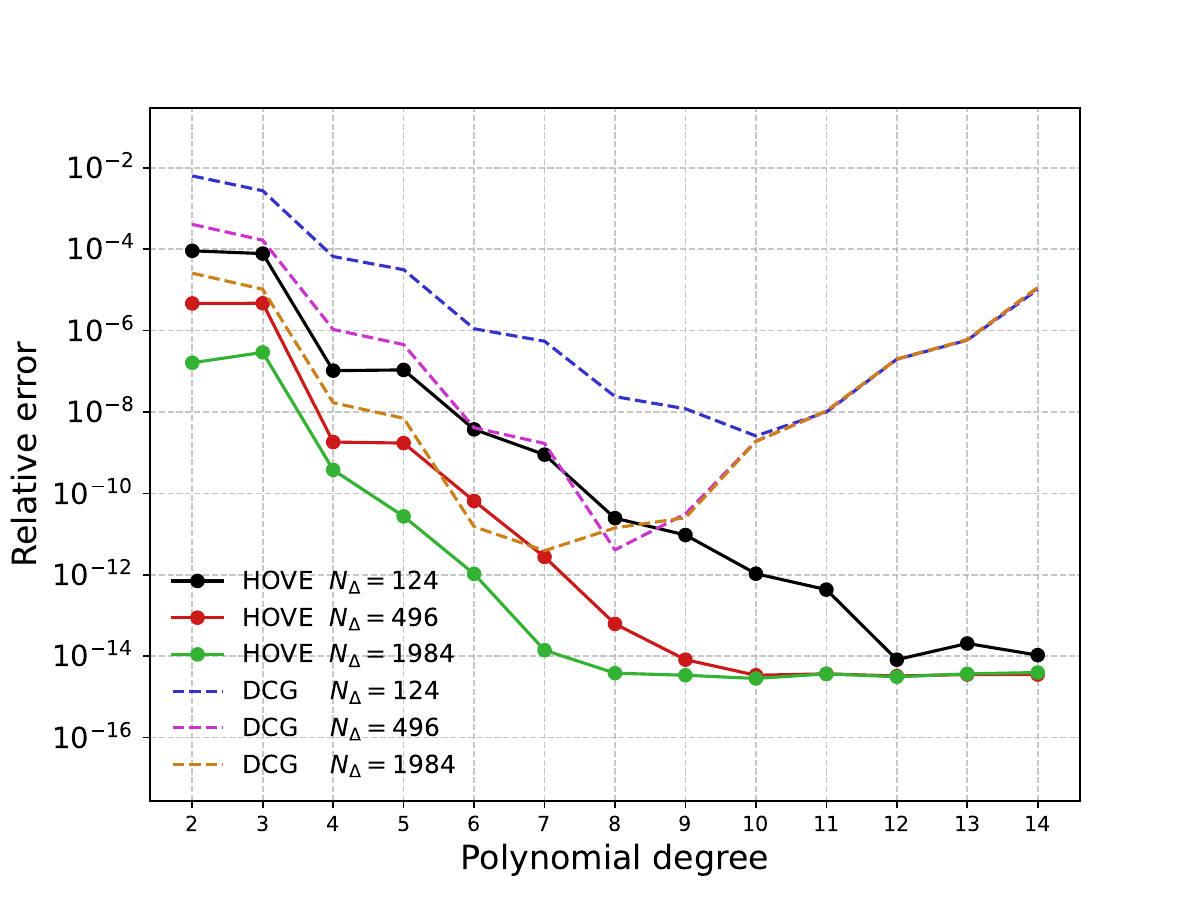}
  \caption{Unit sphere}
  \label{fig:US3}
\end{subfigure}%
\begin{subfigure}{.5\textwidth}
  \centering
 \includegraphics[clip,width=1.\columnwidth]{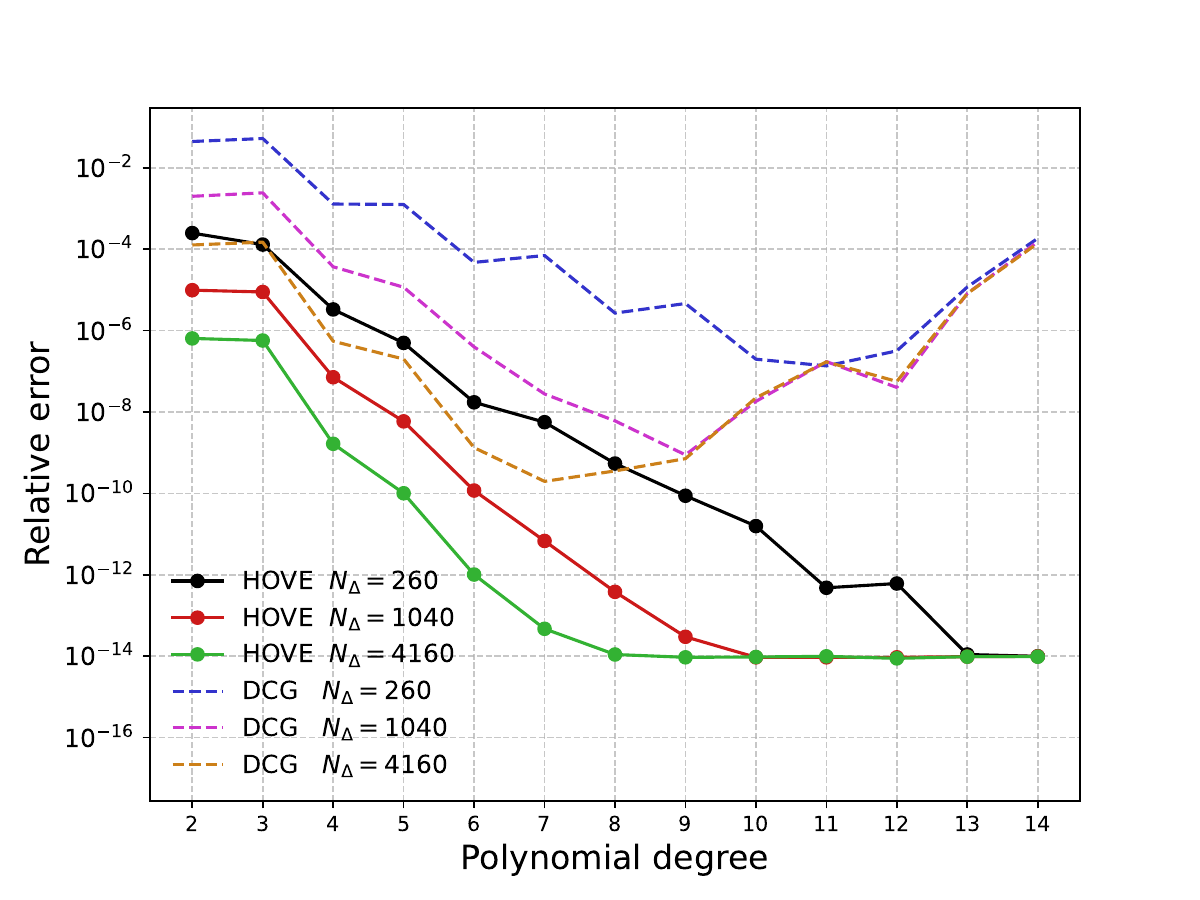}
  \caption{Torus with radii $r=1$ and $R=2$}
  \label{fig:US4}
\end{subfigure}
\caption{Relative errors of DCG and HOVE for surface area of the unit sphere and the torus,
using three different meshes}
\label{fig:R}
\end{figure}
\begin{figure}[t!]
\begin{subfigure}{.5\textwidth}
  \centering
  \includegraphics[clip,width=1\columnwidth]{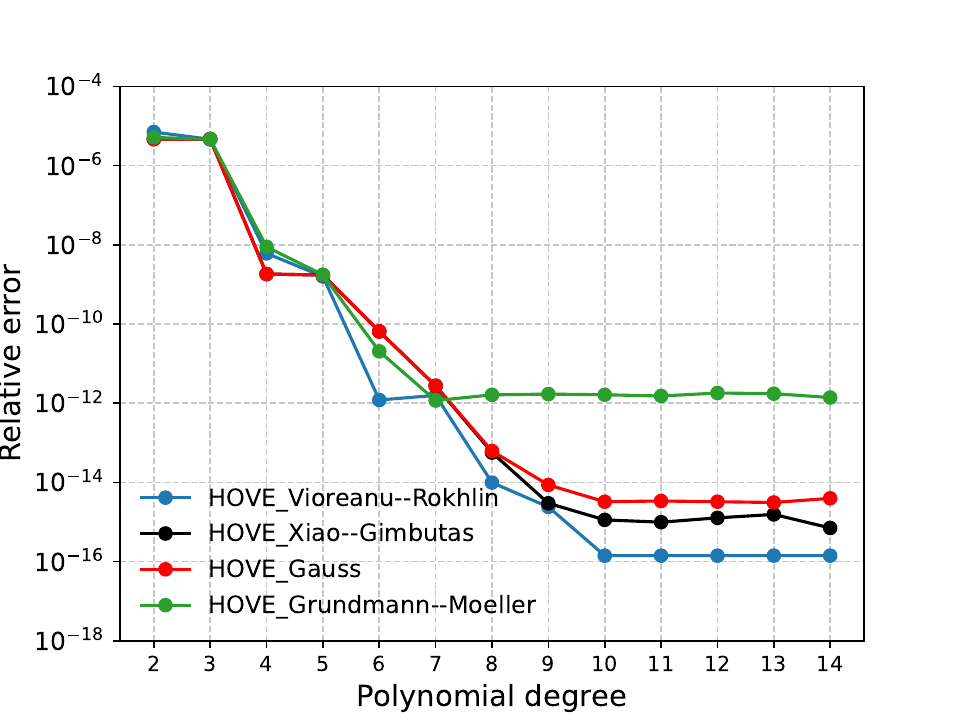}
  \caption{Unit sphere}
  \label{fig:US5}
\end{subfigure}%
\begin{subfigure}{.5\textwidth}
  \centering
 \includegraphics[clip,width=1.\columnwidth]{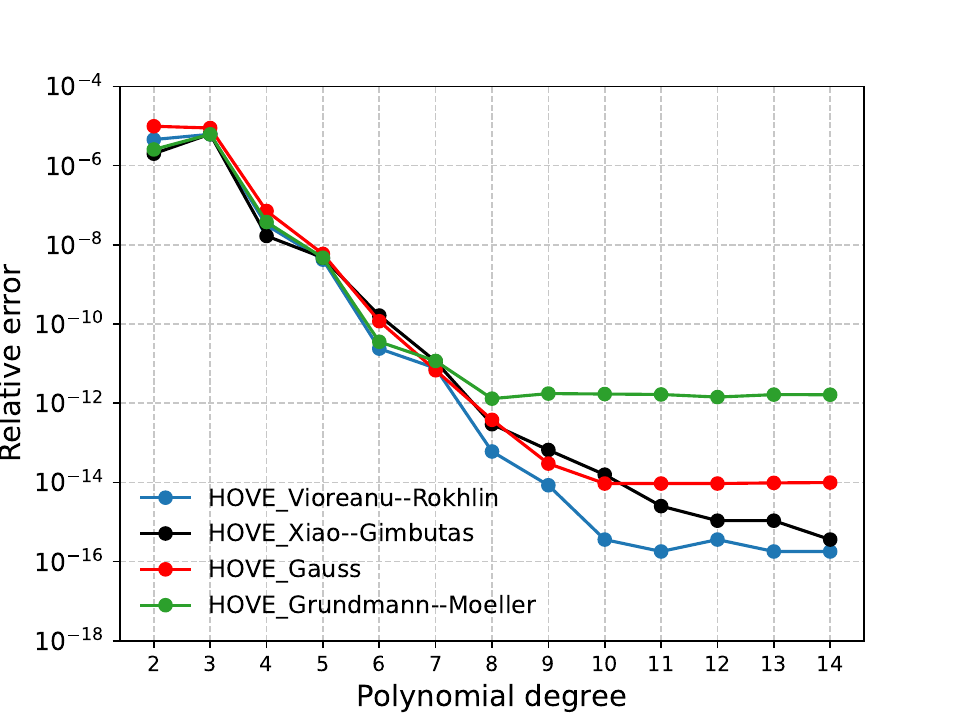}
  \caption{Torus with radii $r=1$ and $R=2$}
  \label{fig:US6}
\end{subfigure}
\caption{Relative errors of HOVE, using Vioreanu–Rokhlin, Xiao–Gimbutas, symmetric Gauss, and Grundmann–Moeller simplex rules, integrating the surface areas of the unit sphere and the torus.}
\label{fig:RE}
\end{figure}

\begin{figure}[t!]
    \begin{minipage}{0.5\textwidth}
        \centering
        \begin{table}[H]
                 \caption{Mesh data}
            \label{tab:IDS+HOSQ}
            \centering
            \begin{tabular}{@{}ccc@{}}
                \toprule
                mesh & \# vertices & \# vertices for IDS \cite{Ray2012} \\
                \midrule
                0 & 272 & 544 \\
                1 & 1088 & 1896 \\
                2 & 4352 & 7528 \\
                3 & 17\,408 & 31\,392 \\
                \bottomrule
            \end{tabular}
        \end{table}
    \end{minipage}
    \begin{minipage}{0.49\textwidth}
        \centering
        \includegraphics[clip,width=0.96\linewidth]{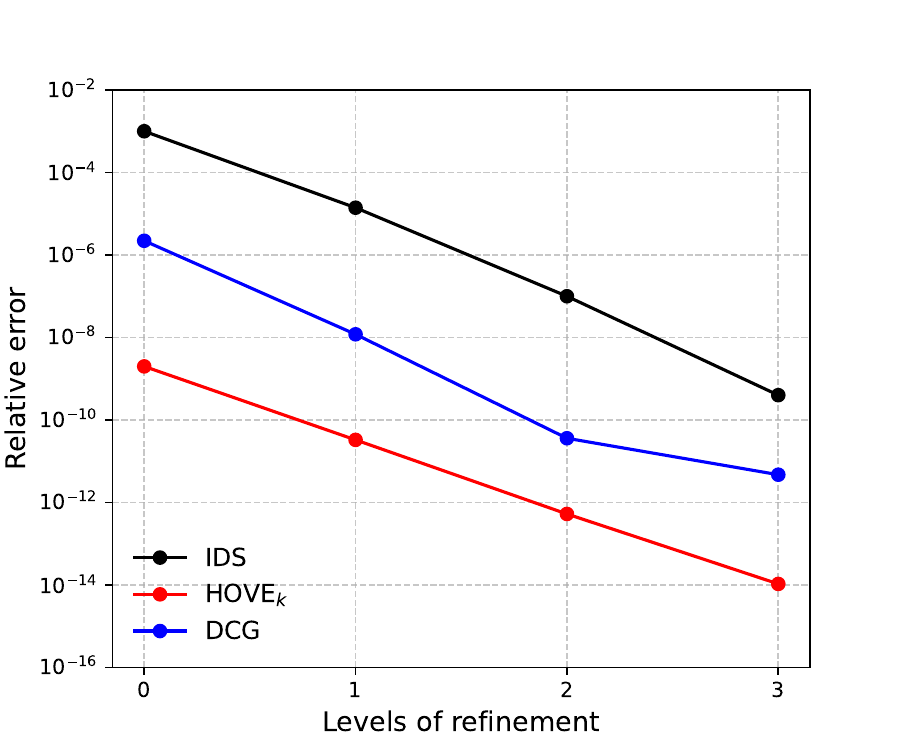}
    \end{minipage}
    \caption{Relative errors of IDS \cite{Ray2012}, DCG, and HOVE$_{k}$ for the surface area of the torus, using a polynomial of degree $6$, are presented on four different meshes, as detailed in the table on the left. }
    \label{fig:IDS}
\end{figure}

The next experiment is the first to involve an actual
integration over a manifold $S$. 
We integrate the
constant function $f=1$ over the unit sphere $\mathrm{S}^2$ and the torus $T^2_{r,R}$ with inner radius $r=1$ and outer radius $R=2$.
The expected result is the surface area, which is
$4\pi$ for the unit sphere and $4\pi^2 rR$
for the torus. We choose initial triangulations of size $N_{\Delta}=124$ for the sphere and of size $N_{\Delta}=260$ for the torus and apply the symmetric Gauss quadrature rule for the triangle $\triangle_2$ of $\deg=14$ with $42$ quadrature nodes~\cite{dunavant1985high}. 

Note that as the integrand $f$ is constant, its
approximation $Q_{G_{d,n}}$ is $f$ itself, and there is
no approximation error.

Fig.~\ref{fig:R} shows the relative errors with respect to the degree of the polynomial interpolation of the geometry.
HOVE stably converges to machine precision with a high algebraic rate, as predicted by Theorem \ref{Main.Thm}.
In contrast, DCG
becomes unstable for orders larger than $\deg =8$. We interpret the instability as the appearance of Runge's phenomenon caused by the choice of equidistant interpolation nodes for DCG.
Indeed, Fig.~\ref{fig:lebesgue_constant} shows a significant difference of the corresponding Lebesgue constants arising for order $k\geq 6$.

Additionally, for each initial mesh, we use {HOVE} with square-squeezing pull-backs of  state-of-the-art simplex quadrature rules (Corollary \ref{cor:pull} \ref{Q2}), including   the symmetric Gauss rule \cite{dunavant1985high}, the Grundmann--Möller quadrature \cite{grundmann1978invariant}, the Xiao--Gimbutas quadrature \cite{xiao2010numerical}, and the Vioreanu--Rokhlin simplex quadrature \cite{vioreanu2014spectra}. Fig.~\ref{fig:RE} shows the relative errors, demonstrating superior accuracy of the HOVE--Vioreanu--Rokhlin rule, but only in the range of machine precision ($10^{-14} \sim 10^{-15}$). However, the Grundmann-Möller quadrature is outperformed by all other rules, which might be attributed to the presence of its negative and positive weights.


\begin{remark}
In the case of the torus  Ray~et~al.~\cite{Ray2012}, conducted a very similar experiment for tori of radii $r=0.7$, $R =1.3$, using the
\emph{High-Order Integration over Discrete Surfaces} (IDS) algorithm \cite[Fig.~5]{Ray2012}, resting on total $l_1$-interpolation degree $k$, with maximum choice $k = 6$.
We perform the same experiment here for  DGC and HOVE with interpolation degree $k =6$,
employing an initial mesh composed of $544$ triangles or equivalently $272$ vertices. We subsequently refine the mesh three times, resulting in similar but coarser meshes than the ones reported by \cite{Ray2012}
Fig.~\ref{fig:IDS} reports the mesh sizes and the relative errors of all methods.

Even though IDS uses meshes of higher resolution, both DCG and HOVE outperform IDS. 
For the rest of this section, we will therefore
disregard the IDS algorithm and only compare 
DCG and HOVE, whereas, for the sake of simplicity, the latter is executed for symmetric or tensorial Gauss rules.
\end{remark}


\begin{figure}[!h]
\begin{subfigure}{.5\textwidth}
  \centering
  \includegraphics[clip,width=0.8\columnwidth]{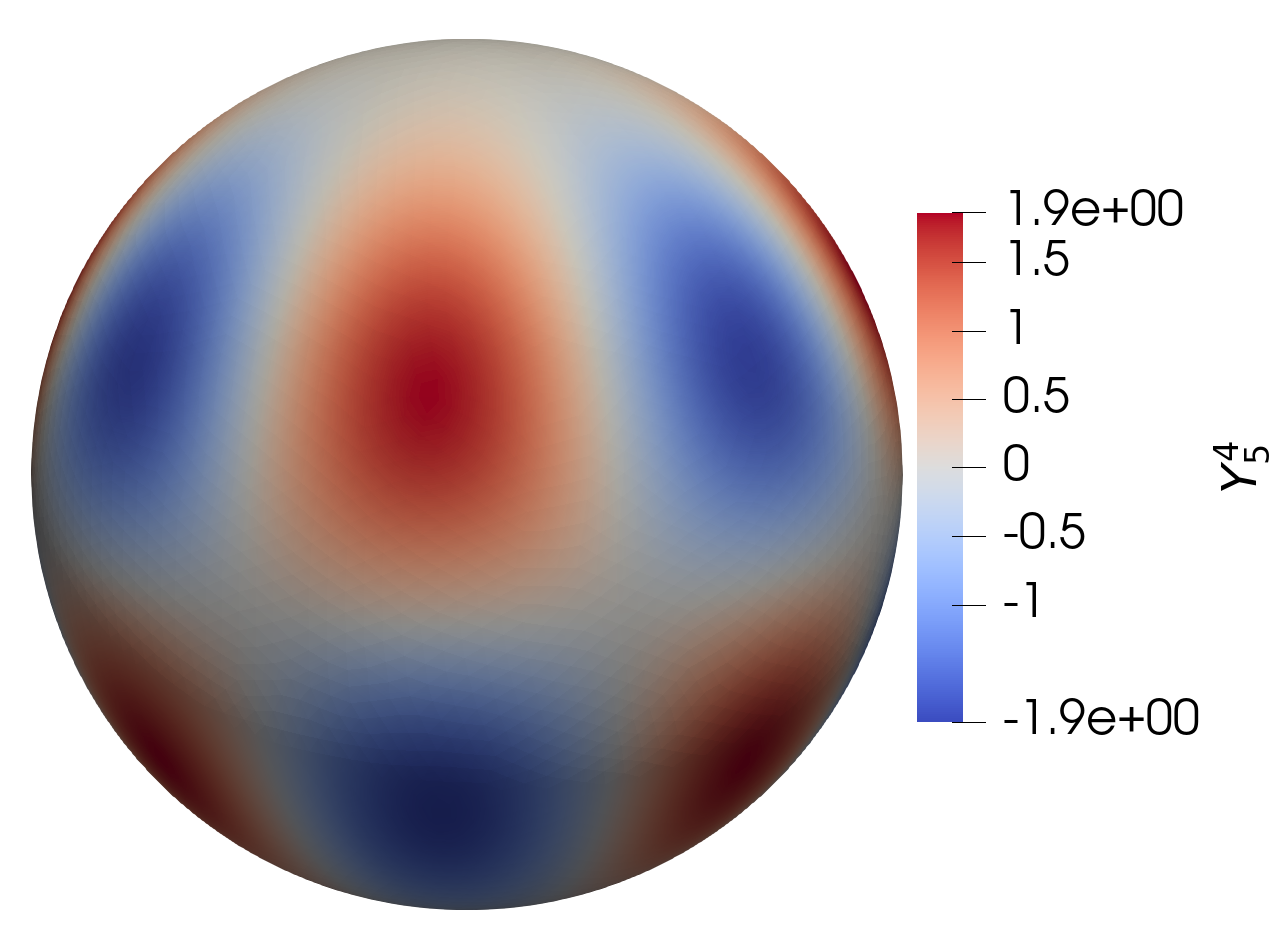}
  $ $ \\
  \hfill
  \caption{Spherical harmonic $Y_5^4$}
  \label{fig:SH1}
\end{subfigure}%
 \hfill
\begin{subfigure}{0.5\textwidth}
  \centering
 \includegraphics[clip,width=1.\columnwidth]{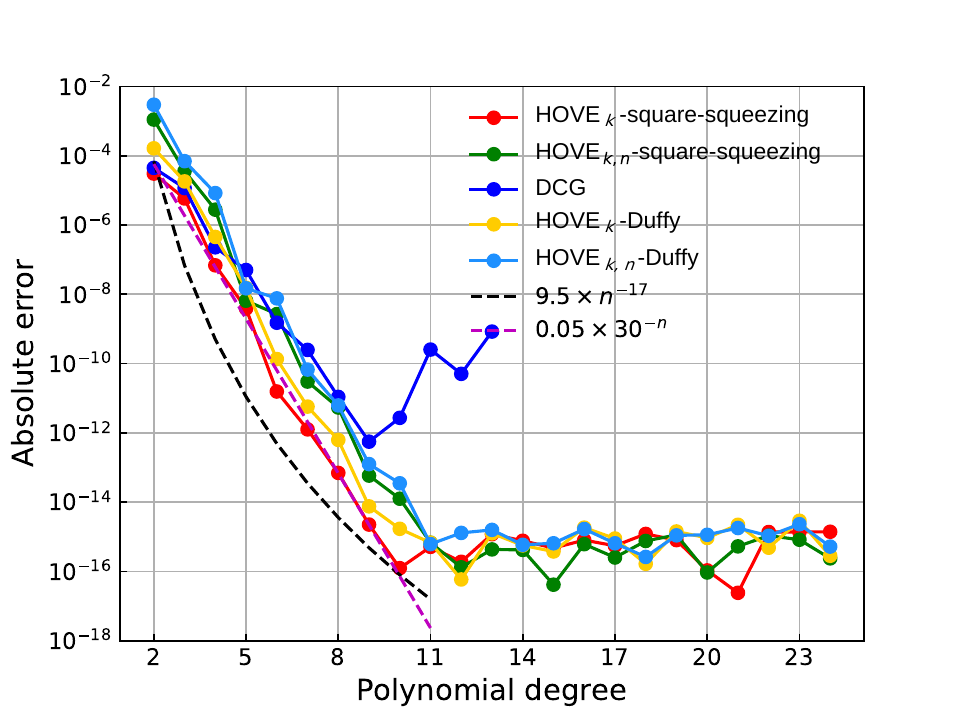}
  \caption{Integration errors}
  \label{fig:SH2}
\end{subfigure}
\caption{Visualization of the spherical harmonic $Y_{5}^{4}$~(left).
Integration errors of DCG and HOVE with respect to the interpolation degree. Abbreviations:  $\text{HOVE}_{k}$ -- interpolating only the geometry, $\text{HOVE}_{k,n}$ -- interpolating  the geometry and the integrand.
\label{fig.duffy_sm}
}
\end{figure}

\subsection{Spherical harmonics}\label{sec:SH}

The next experiment uses a non-constant integrand.
We integrate the $4^{\text{th}}$-order spherical harmonic, visualized in Fig.~\ref{fig.duffy_sm},
over the unit sphere $S^2 \subset \R^3$
\begin{equation*}
 \int_S Y_5^4\,dS
 =0 \,, \quad 
Y^{4}_{5}(x_1,x_2,x_3) 
=
\frac{3\sqrt{385}(x_1^{4}-6x_2^{2}x_1^{2}+x_2^{4})x_3}{16\sqrt{\pi}}\,,
\end{equation*}
vanishing by the $L_2$-orthogonality 
 of the spherical harmonics. 
We approximate the unit sphere by a piecewise flat mesh with
$496$ triangles and compare DCG, HOVE,
and HOVE with the Duffy transformation.
The actual integration is performed using a
symmetric Gauss triangle rule \cite{dunavant1985high} of order $\deg =25$.
Fig.~\ref{fig.duffy_sm}~(right) shows the absolute integration errors as a function of
the polynomial degree for two interpolation scenarios: 

\begin{enumerate}
    \item $\text{HOVE}_{k}$ -- only interpolating the geometry and sampling the integrand directly in the quadrature nodes of a degree-$k$-rule. 
    \item  $\text{HOVE}_{k,n}$ -- interpolating the integrand and the geometry with degree $n=k$ and posterior computing the approximated integral by a degree-$k$-rule.
\end{enumerate}
Both $\text{HOVE}_{k}$ and $\text{HOVE}_{k,n}$ converge with an exponential rates, $0.05\cdot 30^{-n}$ fitted for $\text{HOVE}_{k}$,  as predicted by Theorem~\ref{Main.Thm}. The best fit of an algebraic rate, $9.5\cdot n^{-17}$, does not assert rapid convergence.
%
%
We observe that all three methods behave similarly
for interpolation degrees below~$9$.
For higher degrees, DCG becomes unstable,
whereas HOVE reaches machine precision for degrees above $10$.
HOVE reaches one-order-of-magnitude higher accuracy when utilizing square-squeezing instead of Duffy's transformation.

\subsection{Integrating the Gauss curvature}
\label{sec:experiment_curvature}

\begin{figure}[!htbp]
\begin{subfigure}{.45\textwidth}
\centering
\includegraphics[scale=0.4]{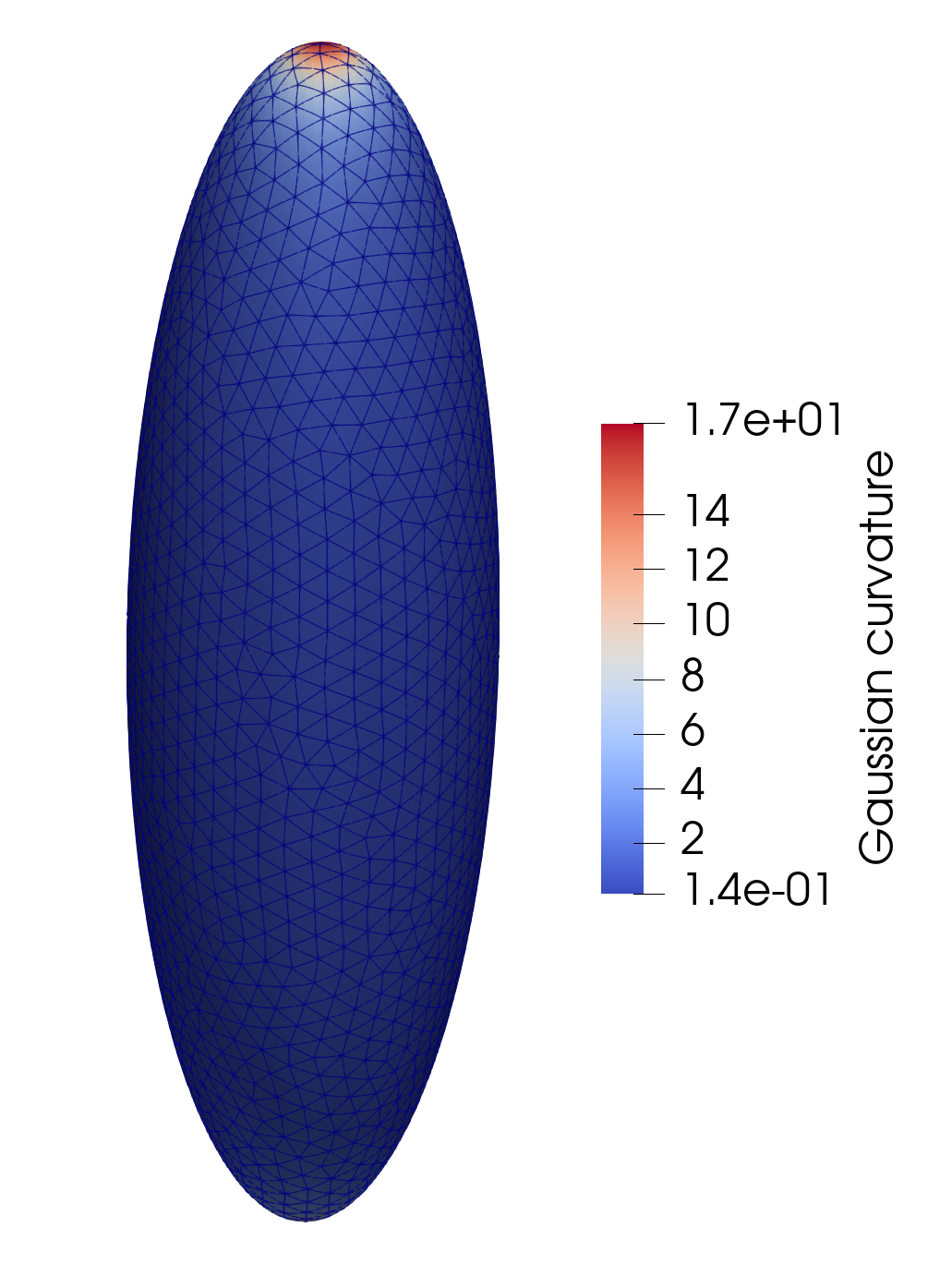}
\caption{Ellipsoid with  $4024$ triangles}
\label{gc.c}
\end{subfigure}%
\begin{subfigure}{.495\textwidth}
 \includegraphics[width=1.0\textwidth]{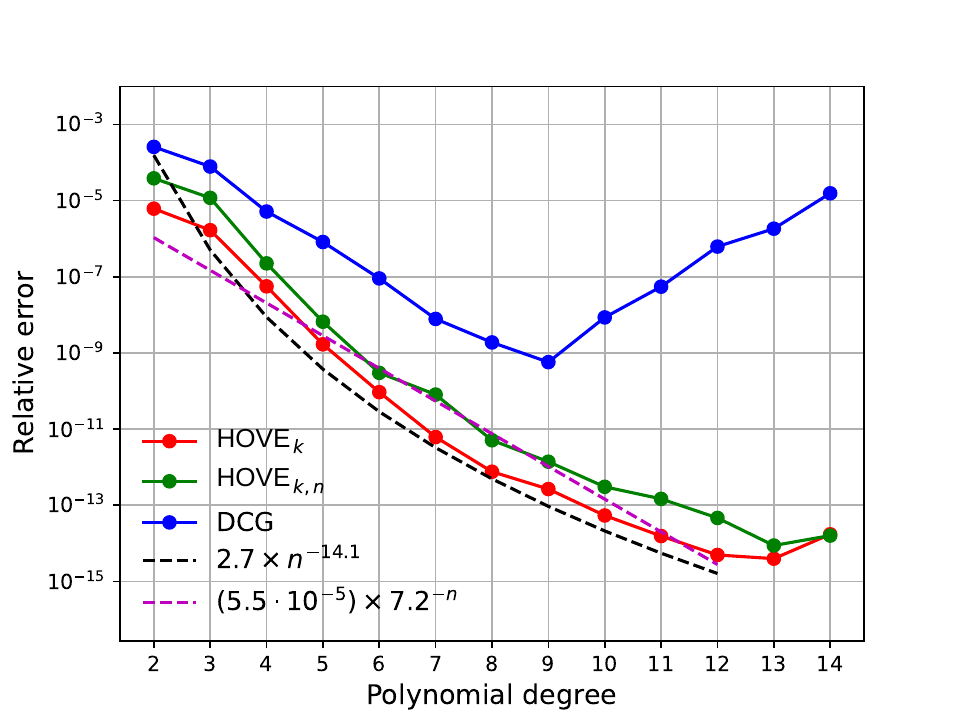}
  \caption{Integration errors }
  \label{r1.c}
\end{subfigure}
\caption{Gauss--Bonnet validation for an ellipsoid with $a=0.6$, $b=0.8$, $c=2$.}
\label{fig:SURF1}
\end{figure}

In this section, we use the Gauss curvature as a non-trivial
integrand. 
By the Gauss--Bonnet theorem \citep{spivak1999,pressley2001}, integrating the Gauss curvature over a closed surface yields
\begin{equation}\label{eq:GB}
    \int_{S}K_{\mathrm{Gauss}}\,dS=2\pi \chi\left(S\right),
\end{equation}
where $\chi\left(S\right)$ denotes the Euler characteristic of the surface. 
We use five surfaces as integration domains.
They are given as the zero sets of the
following five polynomials:

\noindent
\begin{tabular}{lll}
1) & Ellipsoid & $\frac{x^2}{a^2} + \frac{y^2}{b^2} + \frac{z^2}{c^2} = 1$,\quad  $a,b,c \in \R\setminus\{0\}$\\
2) &  Torus & $(x^2 + y^2 + z^2 + R^2 - r^2)^2 - 4R^2(x^2 + y^2) =0$, \quad $0 <r <R \in \R$\\
3) &  Genus 2 surface &  $2y(y^2 - 3x^2)(1 - z^2) + (x^2 + y^2)^2 - (9z^2 - 1)(1 - z^2) = 0$\\
4) & Dziuk's surface & $(x-z^2)^2+y^2+z^2-1 = 0$\\
5) & Double torus & $\big(x^2+y^2)^2-x^2+y^2\big)^2+z^2-a^2=0$ ,\quad  $a \in \R\setminus\{0\}$\\
\end{tabular}

The surfaces, their parameter choices, and the mesh sizes are shown in Fig.~\ref{fig:SURF1}--\ref{fig:SURF5}.
The Gauss curvature is computed symbolically from the
implicit surface descriptions using \linebreak \textsc{Mathematica~11.3}.
HOVE and DCG use (square-squeezing pull-backs of) the symmetric
Gauss simplex rules~\cite{dunavant1985high} of order $14$.

\begin{figure}[!htbp]
\begin{subfigure}{.45\textwidth}
\vfill
\centering
\vspace{-0.5cm}
\includegraphics[scale=0.4]{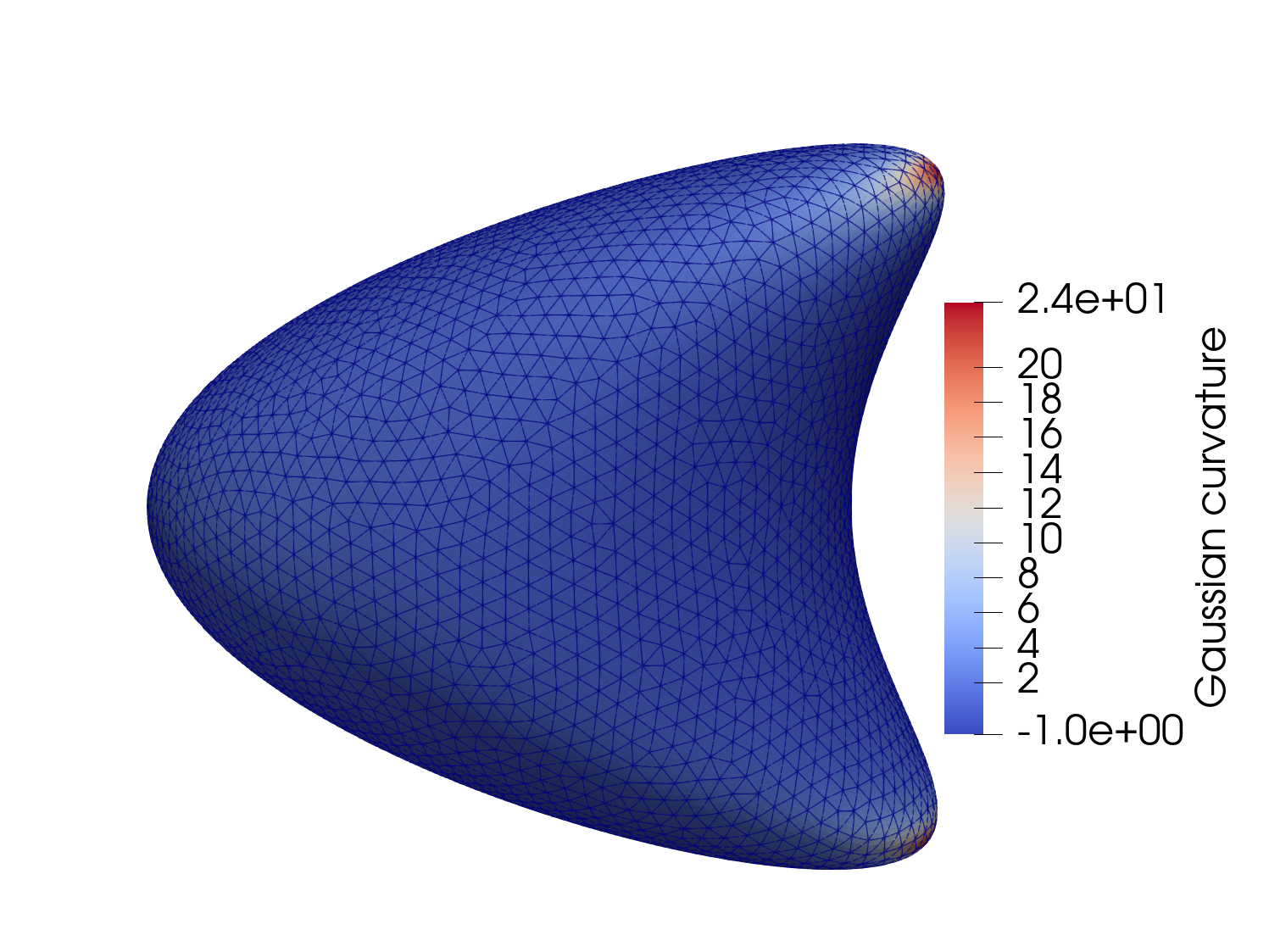}
\caption{Dziuk's surface with $8088$ triangles}
\label{gc.e}
\end{subfigure}
\hfill
\begin{subfigure}{.495\textwidth}
 \includegraphics[width=1.0\textwidth]{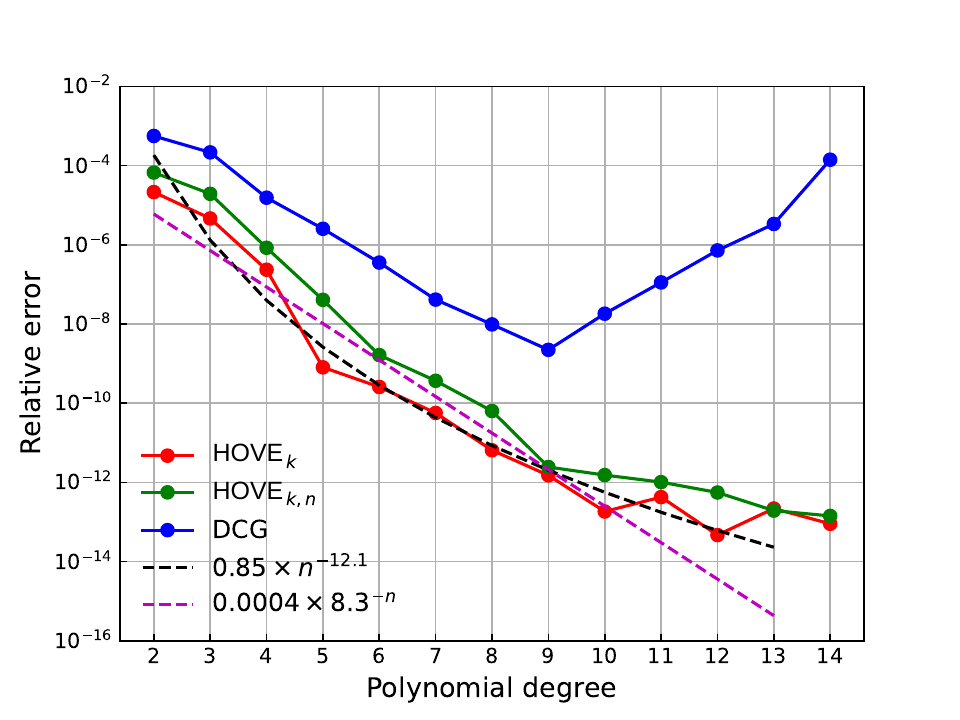}
  \caption{Dziuk's surface}
  \label{r1.f}
\end{subfigure}%
\caption{Gauss--Bonnet validation for Dziuk's surface.}
\label{fig:SURF2}
\end{figure}

\begin{figure}[!htbp]
\begin{subfigure}{.45\textwidth}
\centering
\includegraphics[scale=0.45]{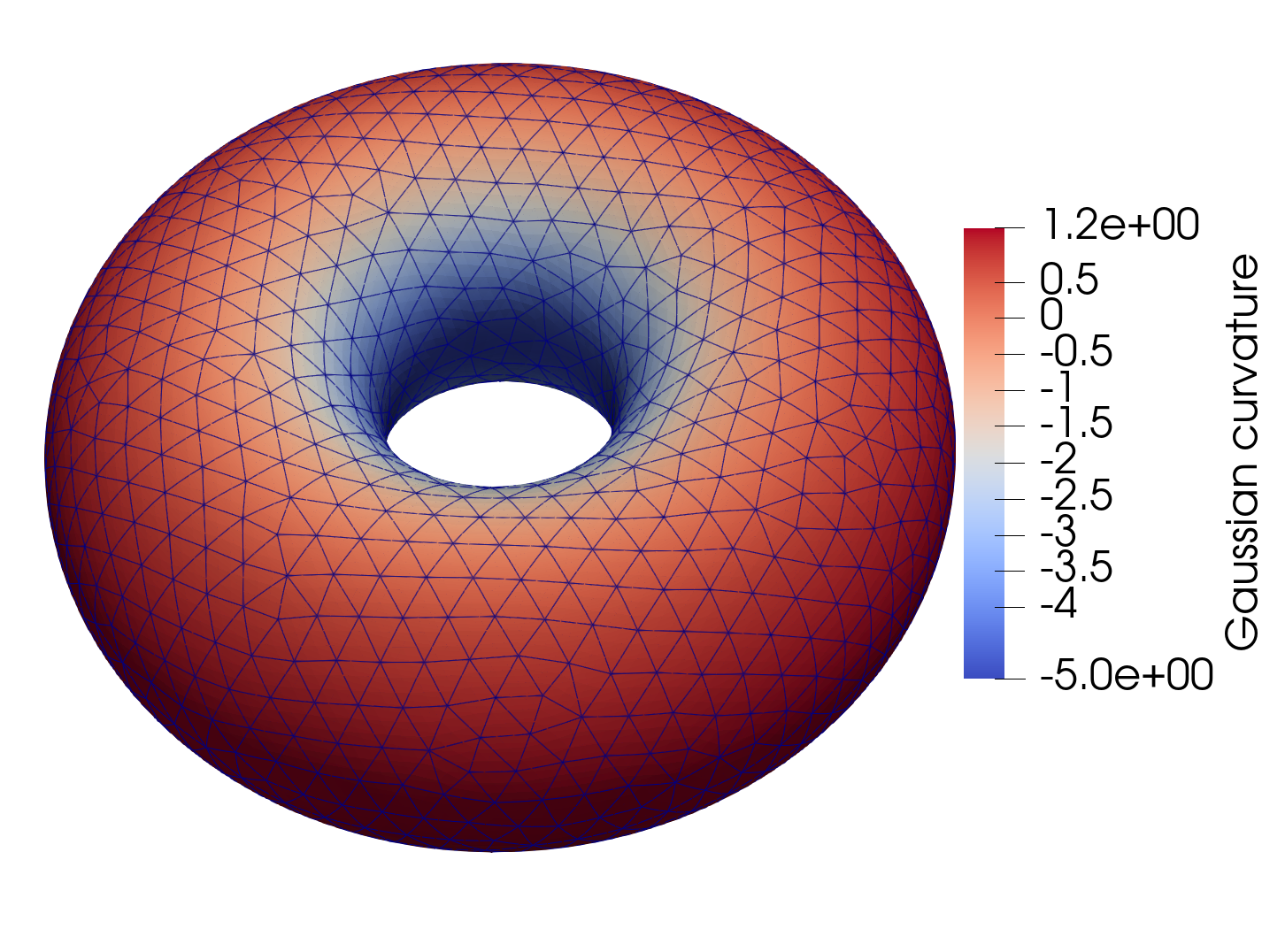}
\caption{Torus  with  $1232$ triangles}
\label{gc.b}
\end{subfigure}
\hfill
\begin{subfigure}{.495\textwidth}
 \includegraphics[width=1.0\textwidth]{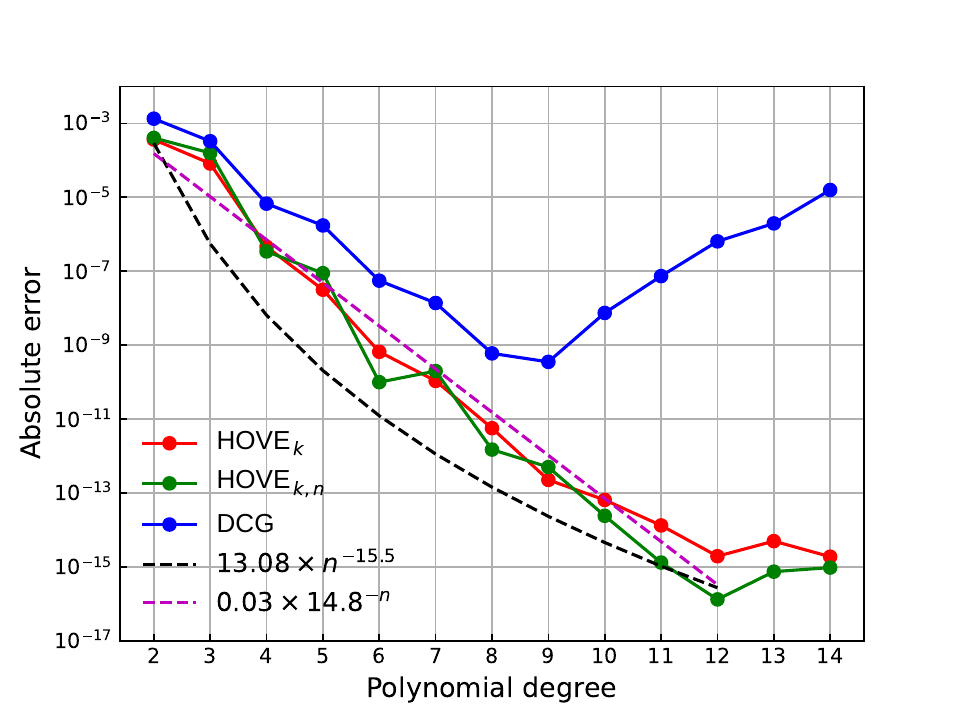}
  \caption{Torus with radii $R=2$, $r=1$}
  \label{r1.d}
\end{subfigure}
\caption{Gauss--Bonnet validation for a torus with radii $R=2$, $r=1$.}
\label{fig:SURF3}
\end{figure}

\begin{figure}[!htbp]
\begin{subfigure}{.45\textwidth}
\centering
\includegraphics[scale=0.45]{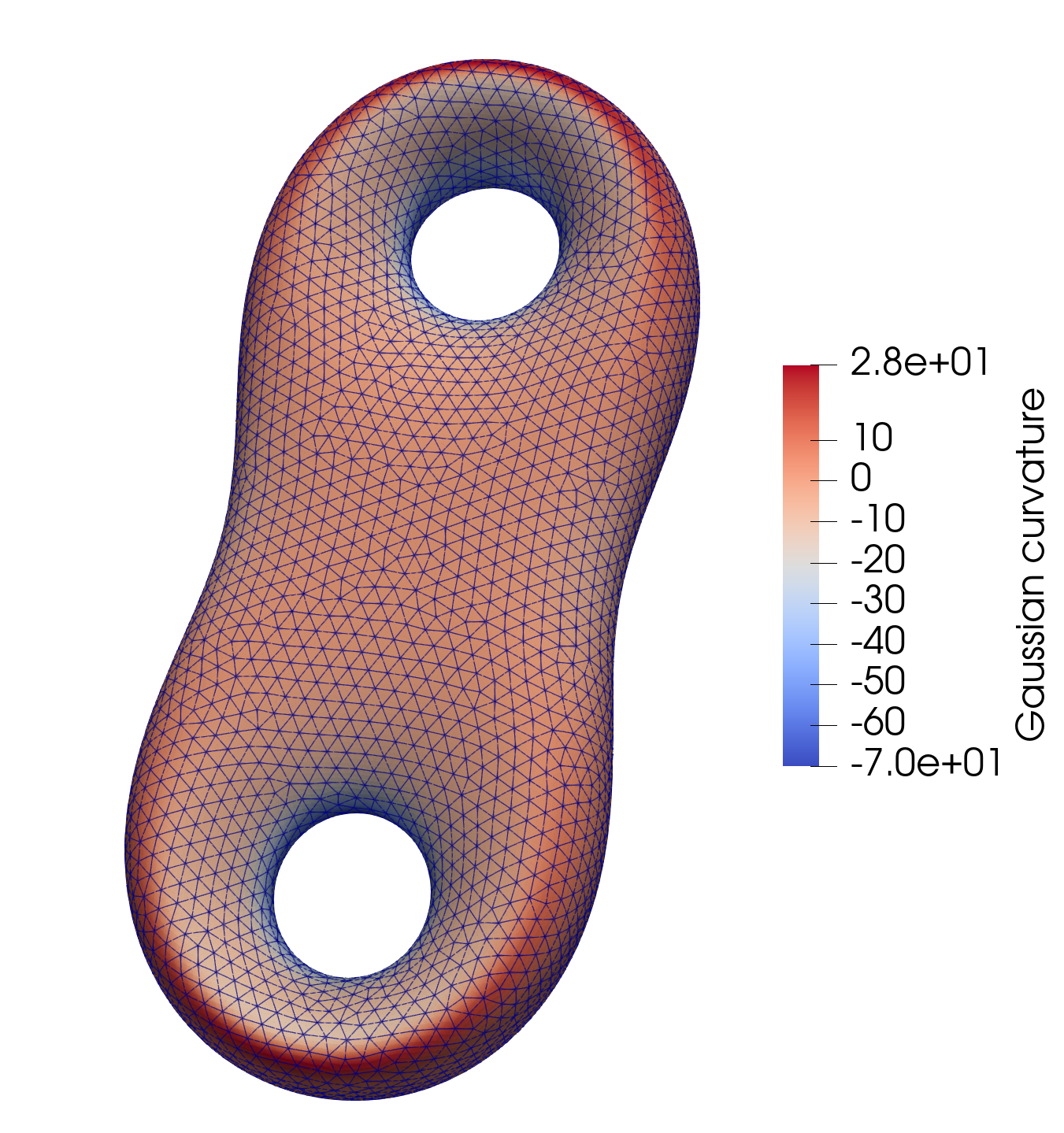}
\caption{Double torus with $8360$ triangles}
\label{gc.f}
\end{subfigure}
\hfill
\begin{subfigure}{.495\textwidth}
 \includegraphics[width=1.0\textwidth]{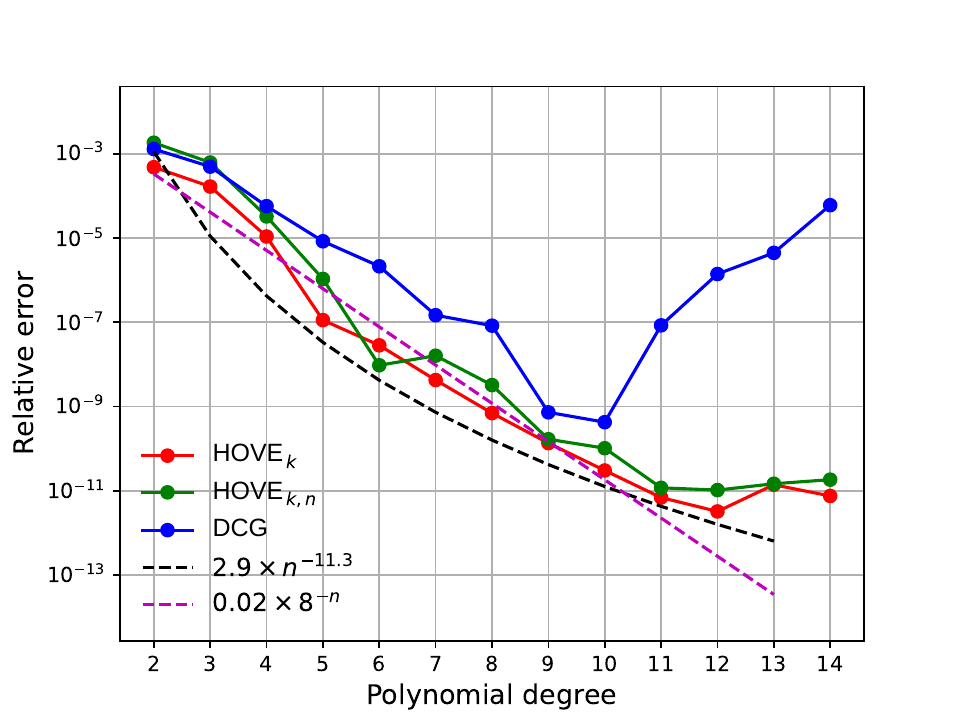}
  \caption{Double torus}
  \label{r1.g}
\end{subfigure}
\caption{Gauss--Bonnet validation for a double torus with $a=0.2$.}
\label{fig:SURF4}
\end{figure}

\begin{figure}[!htbp]
\begin{subfigure}{.45\textwidth}
\centering
\includegraphics[scale=0.45]{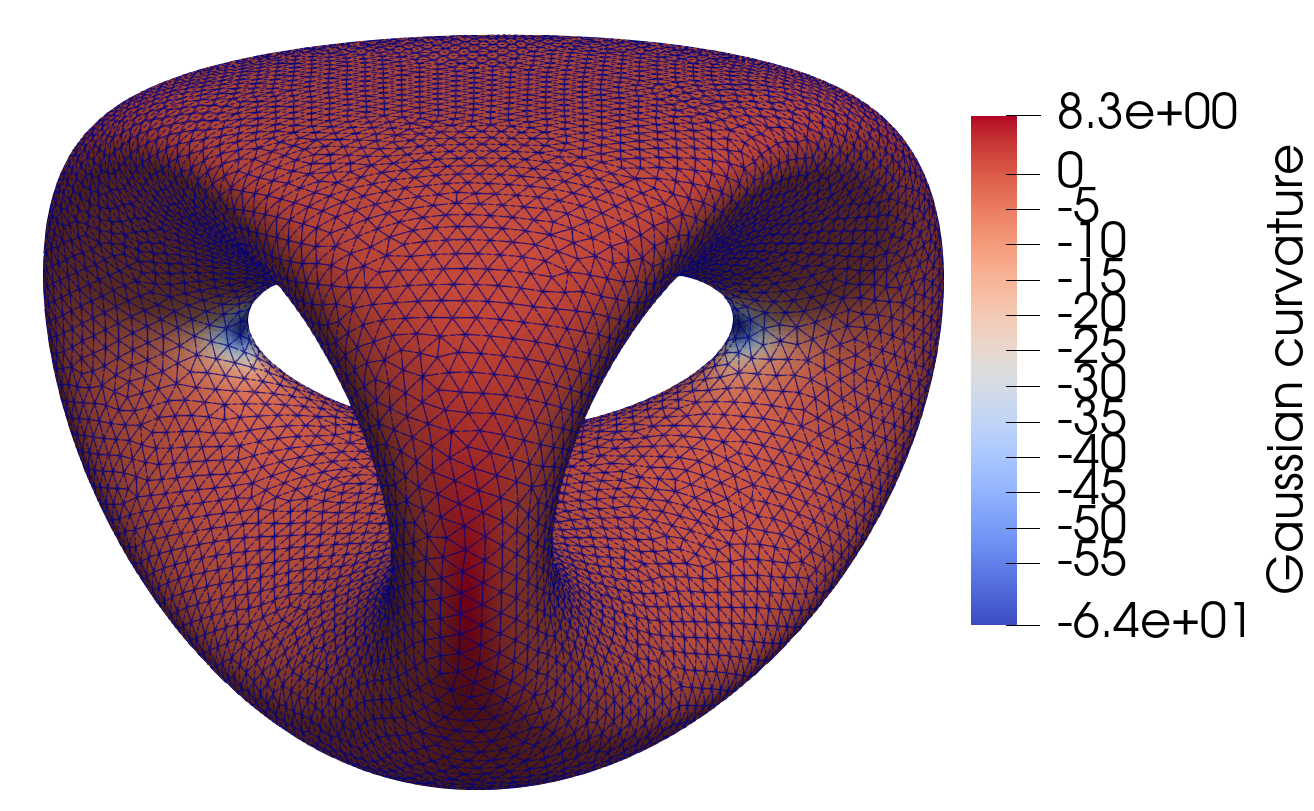}
\caption{Genus~2 surface, with $15\,632$ triangles}
\label{gc.d}
\end{subfigure}
\hfill
\begin{subfigure}{.495\textwidth}
 \includegraphics[width=1.0\textwidth]{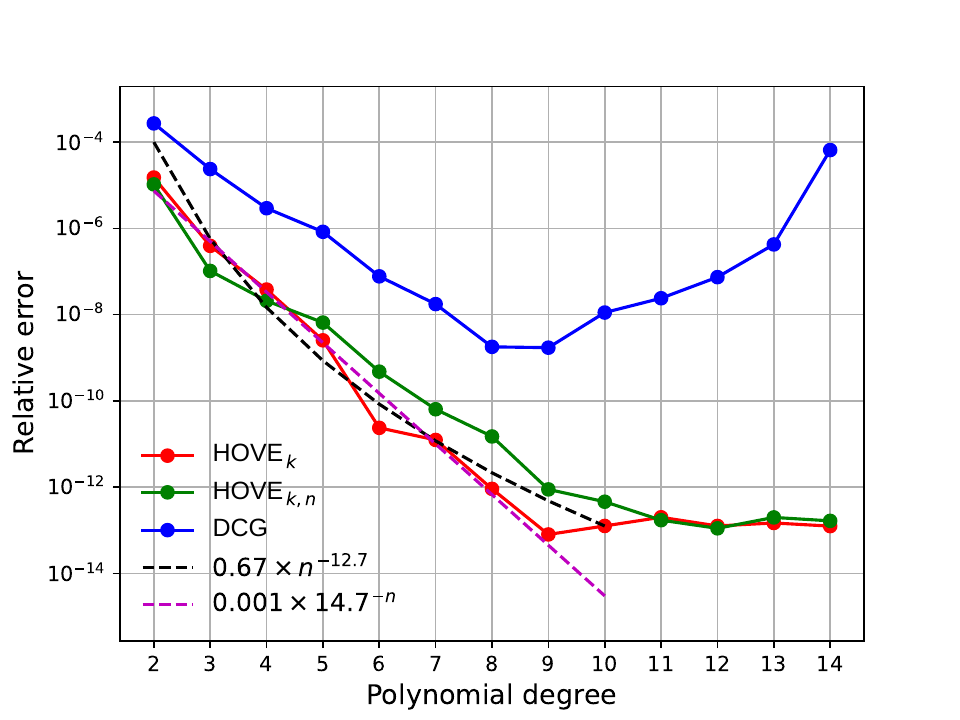}
  \caption{Genus-2 surface}
  \label{r1.e}
\end{subfigure}
\caption{Gauss--Bonnet validation for a genus~2 surface.}
\label{fig:SURF5}
\end{figure}

We keep the experimental design from Section~\ref{sec:SH} and plot the errors as functions of the polynomial degree in Fig.~\ref{fig:SURF1}--\ref{fig:SURF5}. 
Both $\text{HOVE}_{k}$ and $\text{HOVE}_{k,n}$ rapidly converge with exponential rates to the correct value $2\pi\chi(S)$, except for the thin ellipsoid in Fig.~\ref{gc.c}, where both reach super-algebraic rates. In contrast,
DCG fails to reach machine-precision approximations in all of the cases and becomes unstable when using interpolation degrees $k$ larger than $8$.

\subsection{A geometry with a near-singularity}
\label{sec:BC}
\begin{figure}
\begin{subfigure}{.45\textwidth}
\centering
\includegraphics[scale=0.45]{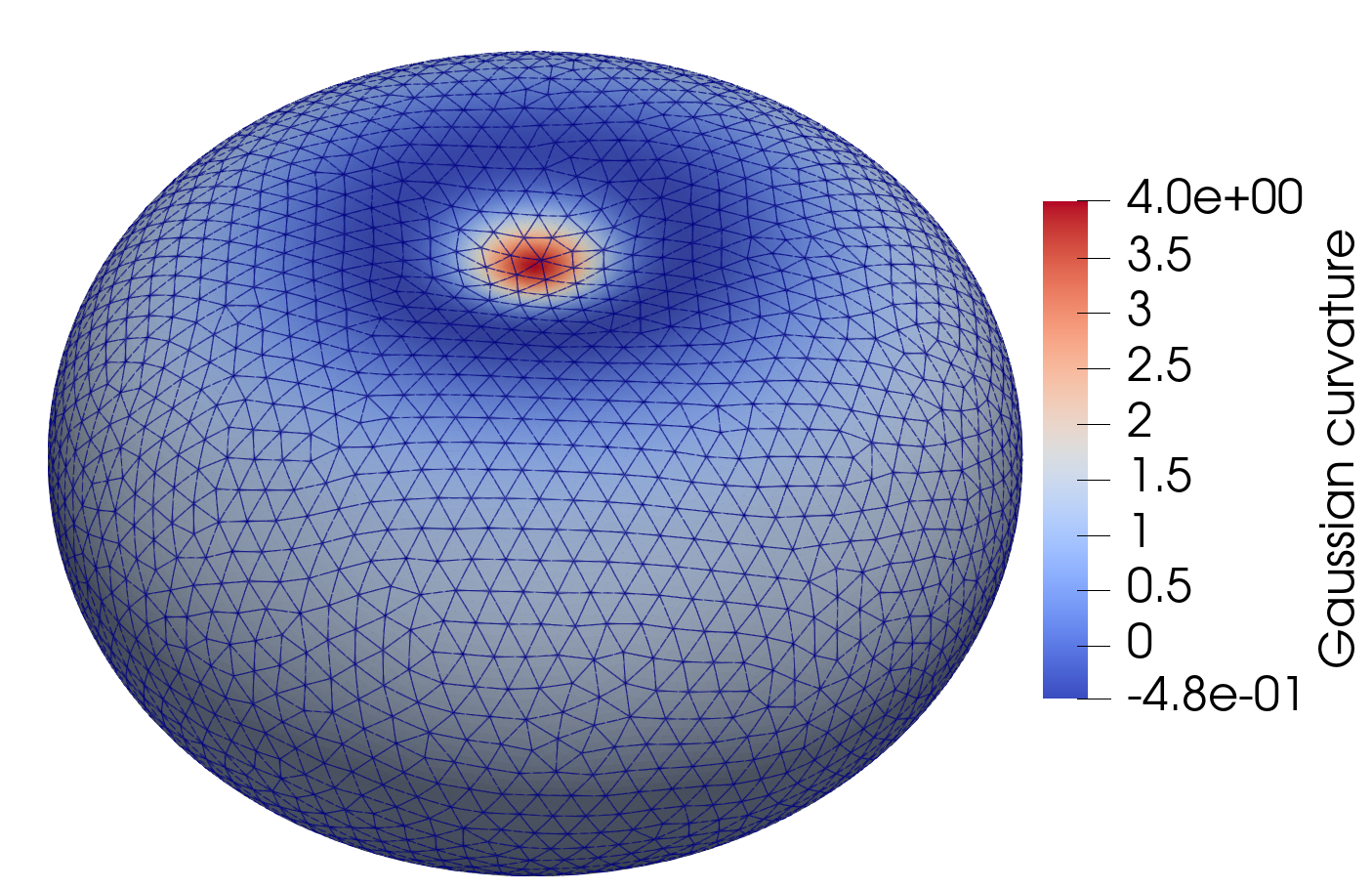}
\caption{Biconcave disc $c=-0.934$, $d=0.80$, with  $5980$ triangles.}
\label{gc.a}
\end{subfigure}%
\hfill
\begin{subfigure}{.5\textwidth}
  \includegraphics[clip,width=1.0\columnwidth]{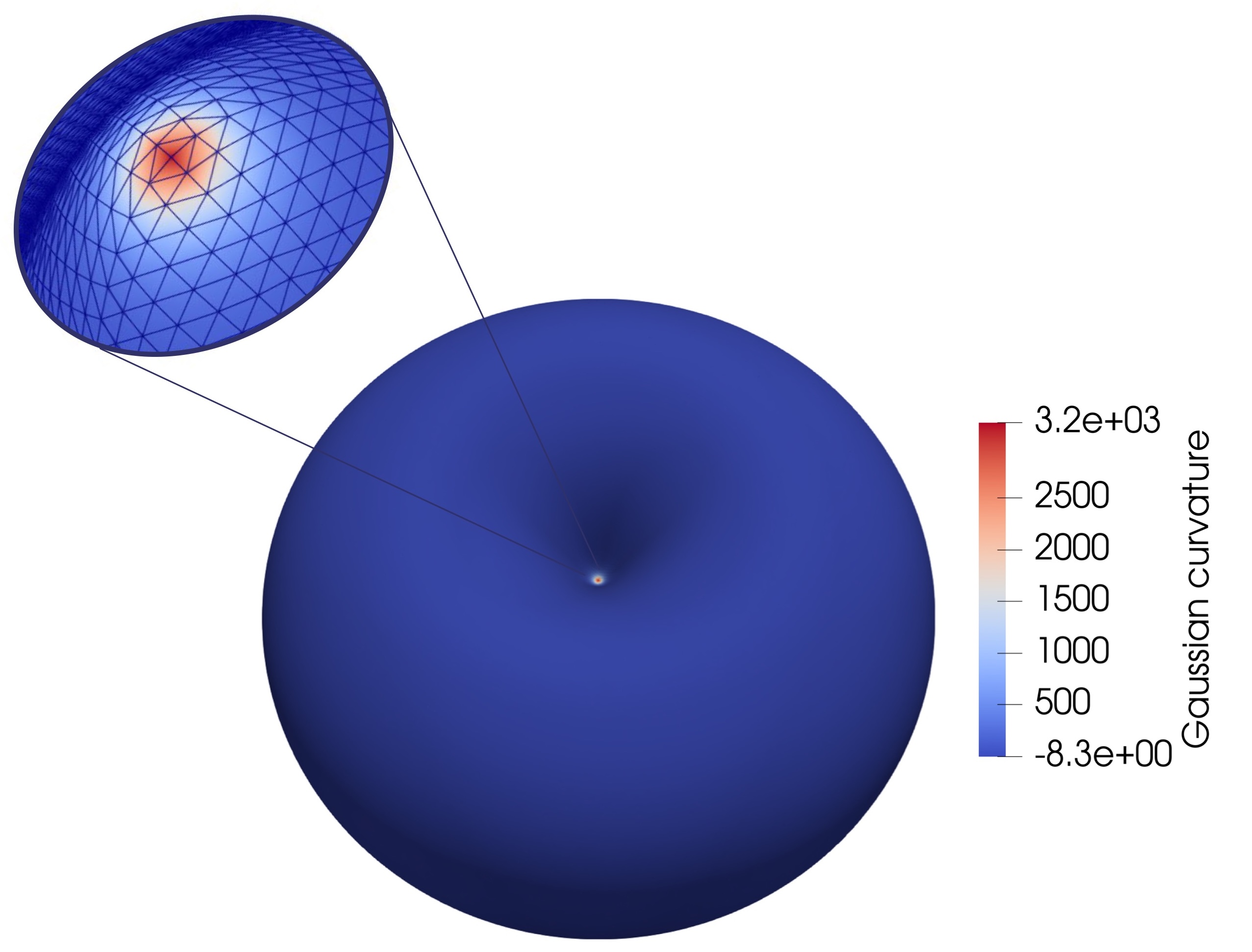}
  \caption{Biconcave disc $c=0.375$, $d=0.5$ with $3144$ triangles.}
  \label{B1}
\end{subfigure}%
\caption{Gauss--Bonnet validation for biconcave discs approaching a vertex singularity. The Gauss curvature
ranges between $[-4.8\cdot10^{-1}]$  and $[4.0]$, Fig.~\ref{gc.a}, and $[-8.3]$
and $[3.2\cdot10^{3}]$, Fig.~\ref{B1}}
\label{fig:biconcave_discs}
\end{figure}

The geometries of the previous section have all been well-behaved.
In contrast, in this section we now test HOVE
on a surface that is close to being singular.
For this, we consider the biconcave discs shown
in Fig.~\ref{fig:biconcave_discs},
which are the zero sets of the polynomial
\begin{equation*}
 P_\text{bicon}(x,y,z) = (d^2 + x^2 + y^2 + z^2)^3 - 8d^2(y^2 + z^2) - c^4,
 \qquad 
 c<d \in \R\setminus\{0\}.\\
\end{equation*}
As long as the parameters $c,d$ are chosen such that $0 \not \in P_\text{bicon}$ the surfaces are smooth. 
We consider the two cases $c=-0.934$, $d=0.8$
and $c=0.375$, $d=0.5$, for which the Gauss curvature
ranges between $[-4.8\cdot10^{-1}]$ and $[4.0]$, and $[-8.3]$
and $[3.2\cdot10^{3}]$, respectively, see Fig.~\ref{B1}. In the latter case, the Gauss curvature increases rapidly by four orders of magnitude when approaching the center, mimicking cone-like singularities \cite{goodwill2021numerical} as a challenge for
high-accuracy integration.


 \begin{figure}[!htp]
\begin{subfigure}{.495\textwidth}
  \includegraphics[width=1.0\textwidth]{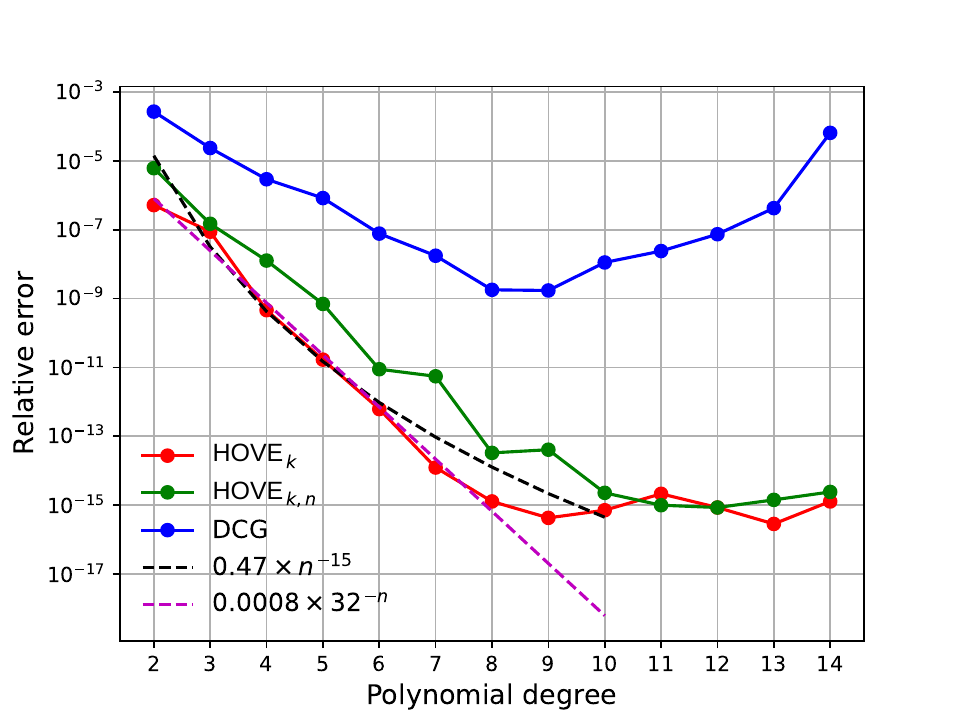}
  \caption{Biconcave disc,  $c=-0.934$, $d=0.8$}
  \label{fig:error_low_curvature_biconcave_disc}
\end{subfigure}%
\begin{subfigure}{.495\textwidth}
  \includegraphics[clip,width=1.0\textwidth]{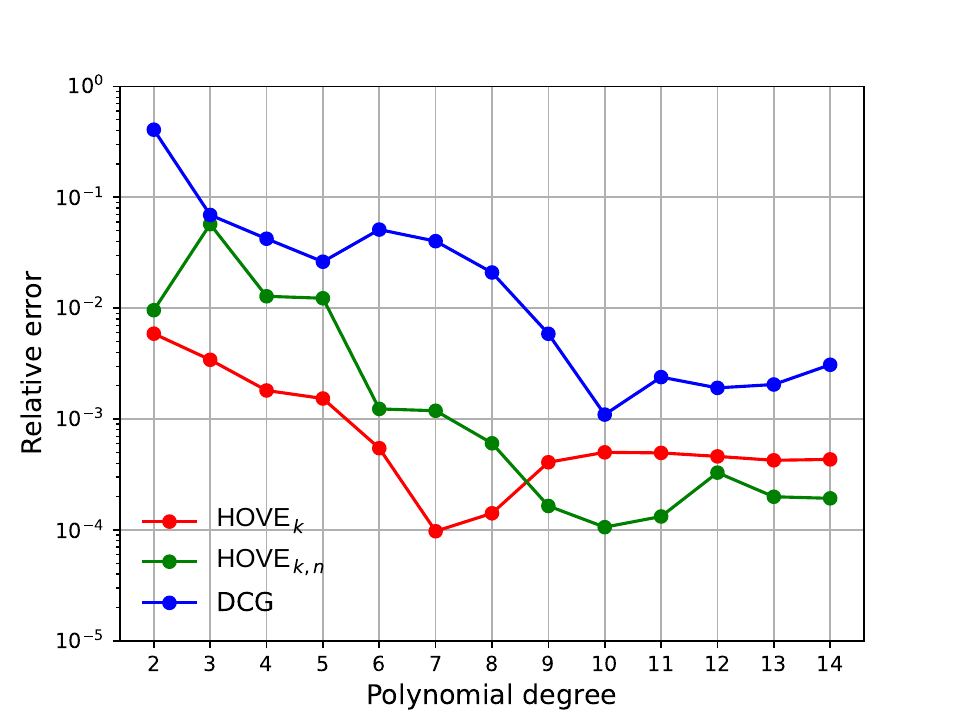}
  \caption{Biconcave disc, $c=0.375$, $d=0.5$}
  \label{fig:error_high_curvature_biconcave_disc}
\end{subfigure}
\caption{Gauss--Bonnet validation for  HOVE and DCG in case of biconcave discs.}
\label{fig:error_biconcave_discs}
\end{figure}

Fig.~\ref{fig:error_low_curvature_biconcave_disc}
shows the Gauss--Bonnet results for the low-curvature case of Fig.~\ref{gc.a}, with $\text{HOVE}_{k}$, $\text{HOVE}_{k,n}$ as in Section~\ref{sec:SH}.
Both HOVE and DCG converge exponentially up to $\deg =9$, but DCG has a slower rate, resulting in  5~orders of magnitude higher precision
for HOVE. For higher orders, the HOVE
error tends to plateau close to a machine precision level.
As in the earlier experiments, DCG becomes unstable in this high-order range.

In the high-curvature case of Fig.~\ref{B1}, none of the approaches reaches machine precision accuracy. 
One may hope that the integration error for such
a near-singular integrand and geometry reduces when applying a mesh $h$-refinement strategy. To test this,
we recompute the integral for the high-curvature case of Fig.~\ref{B1},
on a finer mesh with $50\,304$ triangles.
Fig.~\ref{B2} shows that this leads
to an improvement for HOVE and DCG that, however, still does not reach machine precision accuracy. Notably, HOVE performs up to three orders of magnitude better than DCG and exhibits consistent stability
even for high interpolation degrees.

\begin{figure}[!htp]
\begin{subfigure}{.495\textwidth}
  \centering
 \includegraphics[clip,width=1.0\columnwidth]{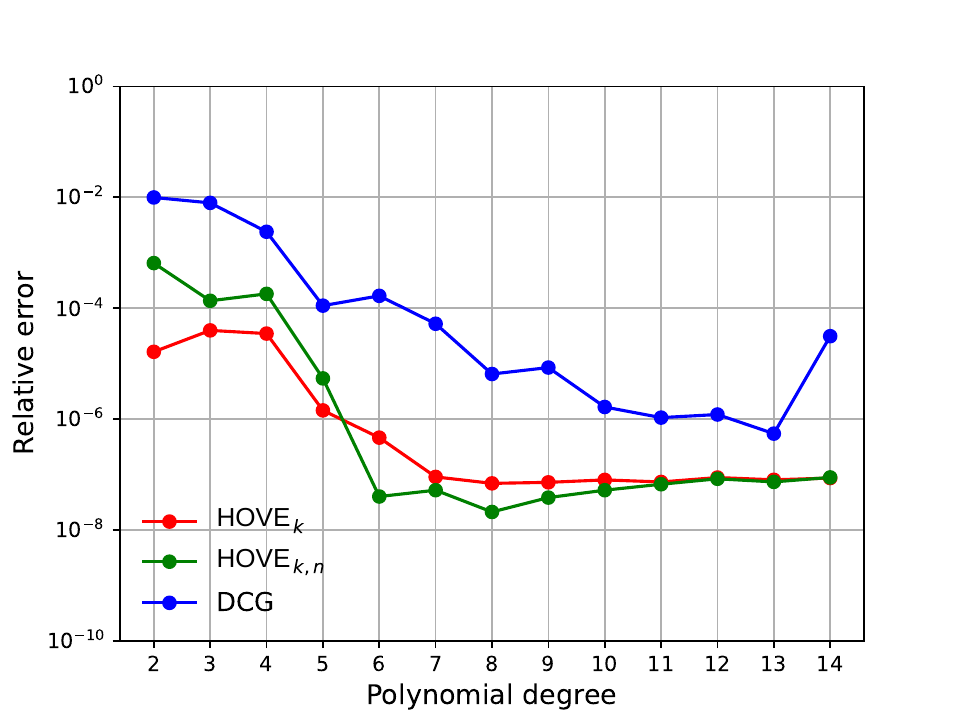}
  \caption{$h-$\text{refinement}}
  \label{B2}
\end{subfigure}
\begin{subfigure}{.495\textwidth}
  \includegraphics[width=1.0\textwidth]{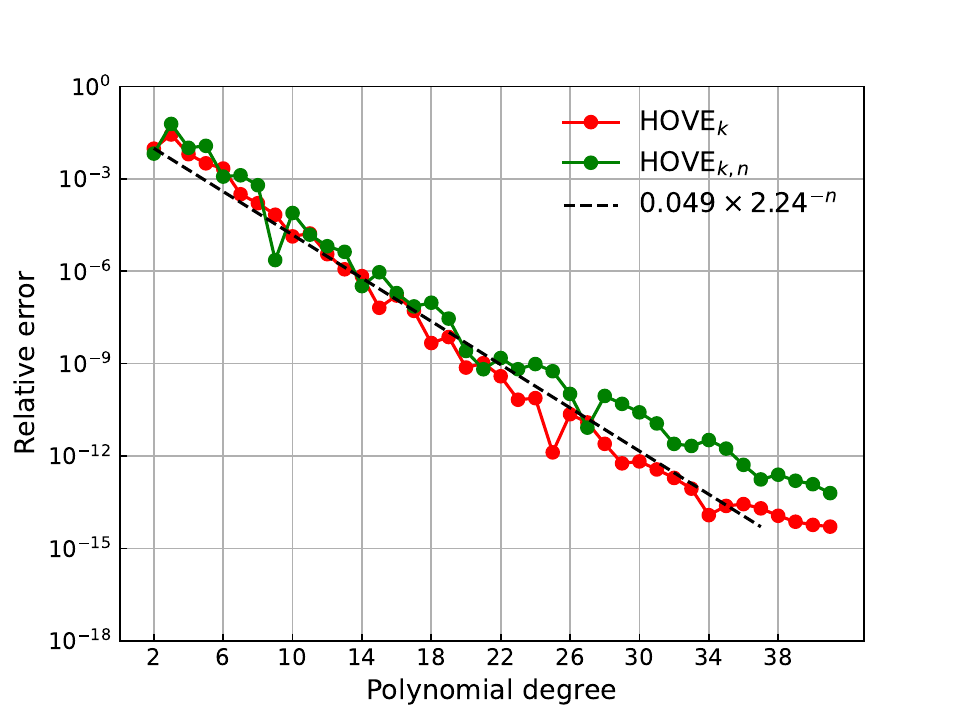}
  \caption{$p-$\text{refinement}}
  \label{B3}
\end{subfigure}
\caption{Gauss--Bonnet validation for the biconcave disc of Fig.~\ref{B1} following a h-refinement strategy with $50\,304$ triangles (\ref{B3}) and a  p-refinement strategy for  $3144$ triangles (\ref{B3}). }
\label{fig:R_bioc}
\end{figure}
As HOVE imposes no restrictions on the polynomial degree, we revisit the initial mesh, consisting of $3144$ triangles, and increase the geometry approximation degree up to $k =1 ,\dots,40$. To simplify the integration process, we employ a tensorial Gauss-Legendre quadrature rule of order $k$ equally to  the  interpolation  degrees $k=n$. 

Fig.~\ref{B3} shows an exponential approximation rate of HOVE until reaching machine precision, conducting the p-refinement. As initially announced in Section~\ref{sec:contr}, this validates HOVE's effectiveness in addressing high variance integration tasks, approaching weak vertex singularities that cannot be resolved by h-refinements.

\begin{figure}[!htp]
\centering
\begin{subfigure}{.5\textwidth}
\centering
\includegraphics[clip,width=0.74\columnwidth]{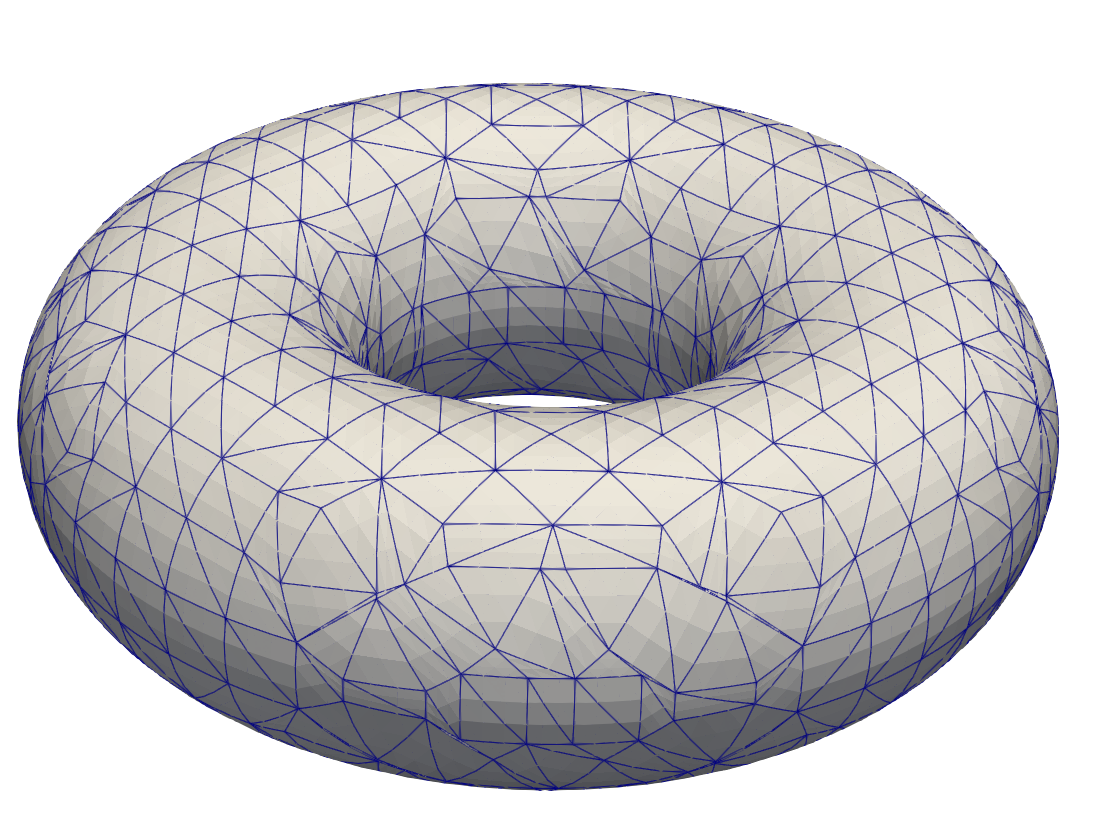}
  \hfill
  \caption{Low-quality mesh with folded triangles}
  \label{f.tri}
\end{subfigure}%
\hfill
\begin{subfigure}{.5\textwidth}
\centering
 \includegraphics[clip,width=0.77\textwidth]{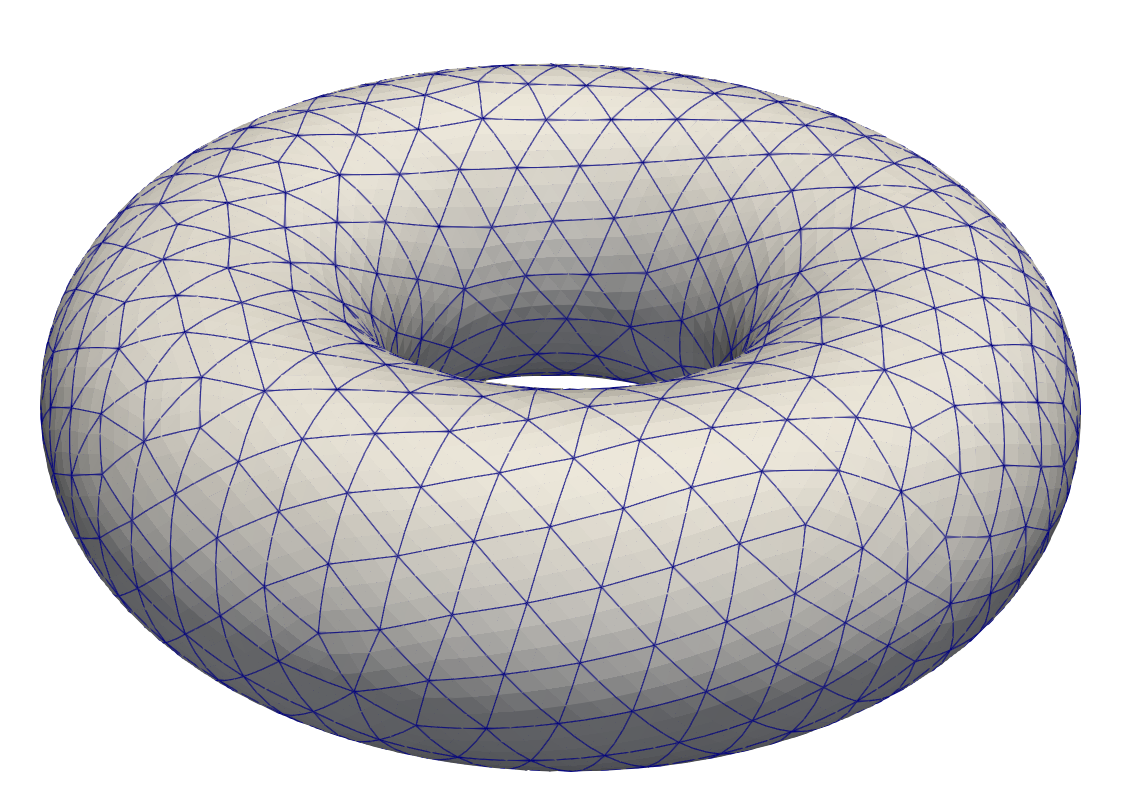}
 \hfill
  \caption{High-quality mesh}
  \label{op.triangles}
\end{subfigure}
\hfill
\begin{subfigure}{.5\textwidth}
\centering
  \includegraphics[clip,width=1\columnwidth]{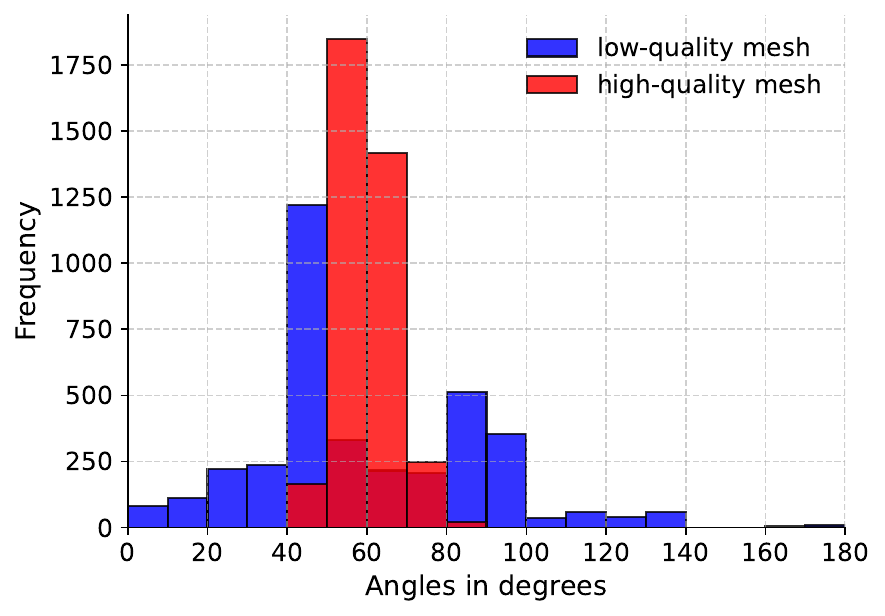}
  \hfill
  \caption{Angles of the two meshes}
  \label{dis.tri}
\end{subfigure}%
\begin{subfigure}{.5\textwidth}
\centering
 \includegraphics[clip,width=1\columnwidth]{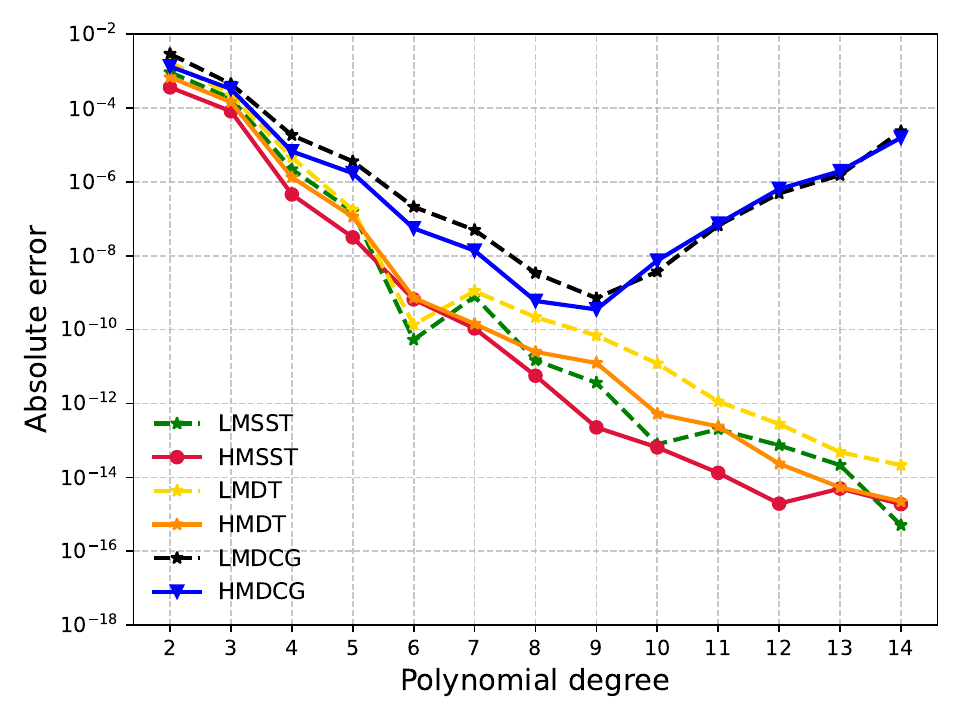}
 \hfill
  \caption{Integration errors}
  \label{r.error}
\end{subfigure}
\hfill
\caption{Integrating the Gauss curvature over
a torus with radii $r=1$, $R=2$ for low-quality and  high-quality meshes consisting of $1232$ triangles. Abbreviations:  LMSST/HMSST -- square-squeezing transform on low/high-quality mesh, LMDT/HMDT --  Duffy's transform on low/high-quality mesh, LMDCG/HMDCG -- DCG on low/high-quality mesh}

\label{fig:compare_meshes}
\end{figure}

\subsection{Mesh quality}

Since the integration error measured here involves in particular the error of approximating the geometry by polynomials it is reasonable to ask whether the integration error depends on the quality of the triangulation of $S$.
To investigate this, we repeat the Gauss--Bonnet validation one final time for the torus of Fig.~\ref{fig:US4}.

We generate two meshes for the torus geometry,
both with $1232$ triangles, shown in
Fig.~\ref{fig:compare_meshes} together with
a plot giving the distributions of the interior
angles. One of the grids is of high quality,
with all angles near $60^\circ$. The second mesh
was deliberately constructed to be of low quality,
featuring a wide range of angles, and even
triangles with inverted orientation.

We evaluate the performance of DCG and HOVE based on square-squeezing and on Duffy's transformation, when exploiting (pull-back) quadrature rules of order $14$.

The results are given in Fig.~\ref{r.error}.
Neither HOVE nor DCG seem to
seriously depend on the mesh quality.
As usual, HOVE converges faster than DCG, and it
converges all the way to the machine precision
limit. DCG shows the same behavior as in all prior experiments, 
becoming unstable for geometry approximation orders
beyond $8$.

\section{Outlook}\label{sec:CON}
For extending the HOVE to integration tasks on non-parameterized surfaces, we aim to use the \emph{global polynomial level set} method (GPLS) \cite{GPLS}, developed by ourselves. GPLS delivers the required  machine--precision--close implicit parameterization $S = l^{-1}(0) $ for a broad class surfaces $S$, being only known in a set of sample nodes.
In combination with the regression techniques in \cite{REG_arxiv} this will enable the computation of surface integrals if in addition, the integrand is only known at a priori given data points.


Our quadrilateral re-parameterizarion due to square-squeezing suggests that the proposed method has the potential to substantially contribute  to triangular spectral element methods (TSEM) \cite{karniadakis2005spectral,heinrichs2001spectral}, realizing fast spectral PDE solvers on surfaces \cite{fortunato2022high}.

\section*{Acknowledgement}
This research
was partially funded by the Center of Advanced Systems Understanding (CASUS), which is financed by Germany’s Federal Ministry of Education and Research (BMBF) and by the Saxon Ministry for Science, Culture, and Tourism (SMWK) with tax funds on the basis of the budget approved by the Saxon State Parliament. The authors gratefully acknowledge the associated editor and the anonymous referees for their valuable comments and constructive suggestions, which helped improve this paper.



\bibliographystyle{siamplain}
\bibliography{Ref.bib}
\end{document}